\documentclass[12pt,leqno]{article}
\usepackage{amsmath}
\usepackage{amsfonts}
\usepackage{amssymb}

\def\thebibliograph#1#2{\section*{{\normalsize \bf #2}}\list
   {[\arabic{enumi}]}{\settowidth\labelwidth{[#1]}\leftmargin\labelwidth
     \advance\leftmargin\labelsep
     \usecounter{enumi}}
     \def\newblock{\hskip .11em plus .33em minus -.07em}
     \sloppy
     \sfcode`\.=1000\relax}

\newtheorem{thm}{Theorem}

\newtheorem{defn}{Definition}

\newtheorem{lem}{Lemma}
\newtheorem{rem}{Remark}

\input{tcilatex}
\begin{document}

\title{Variable Besov spaces: continuous version}
\author{Douadi Drihem}
\date{\today }
\maketitle

\begin{abstract}
We introduce Besov spaces with variable smoothness and integrability by
using the continuous version of Calder\'{o}n reproducing formula. We show
that our space is well-defined, i.e., independent of the choice of basis
functions. We characterize these function spaces by so-called Peetre maximal
functions and we obtain the Sobolev embeddings for these function spaces. We
use these results to prove the atomic decomposition for these spaces.\vskip%
5pt


\end{abstract}

\section{Introduction}

Function spaces play an important role in harmonic analysis, in the theory
of differential equations and in almost\ every\ other\ field of applied
mathematics. Some of these function spaces are Besov spaces. The theory of
these spaces has been developed in detail in \cite{T1} and \cite{T2} (and
continued and extended in the more recent monographs \cite{T3} and \cite{T31}%
), but has a longer history already including many contributors; we do not
want to discuss this here. For general literature on function spaces we
refer to [1, 5, 19, 29, 30-31, 39] and references therein.

Based on continuous characterizations of Besov spaces, we introduce a new
family of function spaces of variable smoothness and integrability. These
type of function spaces, initially appeared in the paper of A. Almeida and
P. H\"{a}st\"{o} \cite{AH},\ where several basic properties were shown, such
as the Fourier analytical characterization, Sobolev embeddings and the
characterization in terms of Nikolskij representations involving sequences
of entire analytic functions. Later, the present author characterized these
spaces by local means and established the atomic characterization (see \cite%
{D3}). After that, Kempka and Vyb\'{\i}ral \cite{KV122} characterized these
spaces by ball means of differences and also by local means, see \cite{N14}
and \cite{IN14} for the duality of these function spaces.\newline

This paper is organized as follows. First we give some preliminaries, where
we fix some notations and recall some basic facts on function spaces with
variable integrability\ and we give some key technical lemmas needed in the
proofs of the main statements. We then define the Besov spaces $\boldsymbol{B%
}_{p(\cdot ),q(\cdot )}^{\alpha (\cdot )}$. We prove a useful
characterization of these spaces based on the so called local means. The
theorem on local means that is proved for Besov spaces of variable
smoothness and integrability is highly technical and its proof is required
new techniques and ideas. Using the results from Sections 3 and 4, we prove
in Section 5 the atomic decomposition for $\boldsymbol{B}_{p(\cdot ),q(\cdot
)}^{\alpha (\cdot )}$.

\section{Preliminaries}

As usual, we denote by $\mathbb{N}_0$ the set of all non-negative integers.
The notation $f\lesssim g$ means that $f\leq c\,g$ for some independent
constant $c$ (and non-negative functions $f$ and $g$), and $f\approx g$
means that $f\lesssim g\lesssim f$. For $x\in \mathbb{R}$, $[x]$ stands for
the largest integer smaller than or equal to $x $.\vskip5pt

If $E\subset {\mathbb{R}^{n}}$ is a measurable set, then $|E|$ stands for
the Lebesgue measure of $E$ and $\chi _{E}$ denotes its characteristic
function. %\vskip5pt
We denote by $\mathcal{S}(\mathbb{R}^{n})$ the space of all rapid decreasing
functions on $\mathbb{R}^{n}$, and by $\mathcal{S}^{\prime }(\mathbb{R}^{n})$
the dual space of $\mathcal{S}(\mathbb{R}^{n})$. The Fourier transform of $%
f\in \mathcal{S}(\mathbb{R}^{n})$ is denoted by $\mathcal{F}f$. %\vskip5pt
By $c$ we denote generic positive constants, which may have different values
at different occurrences. Although the exact values of the constants are
usually irrelevant for our purposes, sometimes we emphasize their dependence
on certain parameters (e.g., $c(p)$ means that $c$ depends on $p$, etc.).
Further notation will be properly introduced whenever needed.\vskip5pt

The variable exponents those we consider are always measurable functions $p$
on $\mathbb{R}^{n}$ with range in $[1,\infty \lbrack$. We denote by $%
\mathcal{P}(\mathbb{R}^{n})$ the set of such functions. We use the standard
notation: 
\begin{equation*}
p^{-}:=\underset{x\in \mathbb{R}^{n}}{\text{ess.inf}}\, p(x) \quad \text{and}
\quad p^+:=\underset{x\in \mathbb{R}^{n}}{\text{ess.sup}}\, p(x).
\end{equation*}
The variable exponent modular is defined by 
\begin{equation*}
\varrho _{p(\cdot )}(f):=\int_{\mathbb{R}^{n}} |f(x)|^{p(x)}\, dx.
\end{equation*}
%where $\varrho _{p}(t)=t^{p}$.
The variable exponent Lebesgue space $L^{p(\cdot )}$\ consists of measurable
functions $f$ on $\mathbb{R}^{n}$ such that $\varrho _{p(\cdot )}(\lambda
f)<\infty $ for some $\lambda >0$. We define the Luxemburg (quasi)-norm on
this space by the formula 
\begin{equation*}
\left\Vert f\right\Vert _{p(\cdot )}:=\inf \left\{\lambda >0:\varrho
_{p(\cdot )}\left(\frac{f}{\lambda }\right)\leq 1\right\}.
\end{equation*}
A useful property is that $\left\Vert f\right\Vert _{p(\cdot )}\leq 1$ if
and only if $\varrho _{p(\cdot )}(f)\leq 1$ (see Lemma 3.2.4 from \cite{DHHR}%
). \vskip5pt

A weight function (a weight) is a measurable function $w: \mathbb{R}%
^{n}\rightarrow \mathbb{(}0,\infty \mathbb{)}$. Given a weight $w$, we
denote by $L^{p(\cdot )}(w)$ the space of all measurable functions $f$ on $%
\mathbb{R}^{n}$ such that 
\begin{equation*}
\left\Vert f\right\Vert _{L^{p(\cdot )}(w)}:=\big\|fw^{\frac{1}{p(\cdot )}}%
\big\|_{p(\cdot )}<\infty.
\end{equation*}%
Let $p,q\in \mathcal{P}(\mathbb{R}^{n})$. The mixed Lebesgue-sequence space $%
\ell _{>}^{q(\cdot )}(L^{p(\cdot )})$ is defined on sequences of $L^{p(\cdot
)}$-functions by the modular 
\begin{equation*}
\varrho _{\ell _{>}^{q(\cdot )}(L^{p\left( \cdot \right)
})}((f_{v})_{v}):=\sum\limits_{v=1}^{\infty }\inf \left\{\lambda
_{v}>0:\varrho _{p(\cdot )}\left(\frac{f_{v}}{\lambda _{v}^{1/q(\cdot )}}%
\right)\leq 1\right\}.
\end{equation*}%
The (quasi)-norm is defined from this as usual:%
\begin{equation}
\left\Vert \left( f_{v}\right) _{v}\right\Vert _{\ell _{>}^{q(\cdot
)}(L^{p\left( \cdot \right) })}:=\inf \left\{\mu >0:\varrho _{\ell ^{q(\cdot
)}(L^{p(\cdot )})}\left(\frac{1}{\mu }(f_{v})_{v}\right)\leq 1\right\}.
\label{mixed-norm}
\end{equation}%
If $q^{+}<\infty $, then we can replace $\mathrm{\eqref{mixed-norm}}$ by a
simpler expression: 
\begin{equation*}
\varrho _{\ell _{>}^{q(\cdot )}(L^{p(\cdot
)})}((f_{v})_{v}):=\sum\limits_{v=1}^{\infty }\left\Vert |f_{v}|^{q(\cdot
)}\right\Vert _{\frac{p(\cdot )}{q(\cdot )}}.
\end{equation*}
The case $p=\infty $ can be included by replacing the last modular by 
\begin{equation*}
\varrho _{\ell _{>}^{q(\cdot )}(L^{\infty
})}((f_{v})_{v}):=\sum\limits_{v=1}^{\infty }\big\|\left\vert
f_{v}\right\vert ^{q(\cdot )}\big\|_{\infty }.
\end{equation*}

We say that a real valued-function $g$ on $\mathbb{R}^{n}$ is \textit{%
locally }log\textit{-H\"{o}lder continuous} on $\mathbb{R}^{n}$, abbreviated 
$g\in C_{\text{loc}}^{\log }(\mathbb{R}^{n})$, if for any compact set $K$ of 
$\mathbb{R}^n$, there exists a constant $c_{\log }(g)>0$ such that 
\begin{equation}  \label{lo-log-Holder}
\left\vert g(x)-g(y)\right\vert \leq \frac{c_{\log }(g)}{\log
(e+1/\left\vert x-y\right\vert )}
\end{equation}%
for all $x,y\in K$. If 
\begin{equation*}
|g(x)-g(0)|\leq \frac{c_{\log }(g)}{\log (e+1/|x|)}
\end{equation*}%
for all $x\in \mathbb{R}^{n}$, then we say that $g$ is \emph{$\log $-H\"{o}%
lder continuous at the origin} (or has a \emph{logarithmic decay at the
origin}). We say that $g$ satisfies the log\textit{-H\"{o}lder decay
condition}, if there exist two constants $g_{\infty }\in \mathbb{R}$ and $%
c_{\log }>0$ such that%
\begin{equation*}
\left\vert g(x)-g_{\infty }\right\vert \leq \frac{c_{\log }}{\log
(e+\left\vert x\right\vert )}
\end{equation*}%
for all $x\in \mathbb{R}^{n}$. We say that $g$ is \textit{globally} log%
\textit{-H\"{o}lder continuous} on $\mathbb{R}^{n}$, abbreviated $g\in
C^{\log }(\mathbb{R}^{n})$, if it is\textit{\ }locally log-H\"{o}lder
continuous on $\mathbb{R}^n$ and satisfies the log-H\"{o}lder decay\textit{\ 
}condition.\textit{\ }The constants $c_{\log }(g)$ and $c_{\log }$ are
called the \textit{locally }log\textit{-H\"{o}lder constant } and the log%
\textit{-H\"{o}lder decay constant}, respectively\textit{.} We note that any
function $g\in C_{\text{loc}}^{\log }(\mathbb{R}^{n})$ always belongs to $%
L^{\infty }$.\vskip5pt

We define the following class of variable exponents: 
\begin{equation*}
\mathcal{P}^{\mathrm{log}}(\mathbb{R}^{n}):=\left\{p\in \mathcal{P}(\mathbb{R%
}^{n}):\frac{1}{p}\in C^{\log }(\mathbb{R}^{n})\right\},
\end{equation*}%
which is introduced in \cite[Section 2]{DHHMS}. We define 
\begin{equation*}
\frac{1}{p_{\infty }}:=\lim_{|x|\rightarrow \infty } \frac{1}{p(x)},
\end{equation*}
and we use the convention $\frac{1}{\infty }=0$. Note that although $\frac{1%
}{p}$ is bounded, the variable exponent $p$ itself can be unbounded. We put 
\begin{equation*}
\Psi \left( x\right) :=\sup_{\left\vert y\right\vert \geq \left\vert
x\right\vert }\left\vert \varphi \left( y\right) \right\vert
\end{equation*}
for $\varphi \in L^{1}$. We suppose that $\Psi \in L^{1}$. Then it was
proved in $\mathrm{\cite[Lemma \ 4.6.3]{DHHR}}$ that if $p\in \mathcal{P}^{%
\mathrm{log}}(\mathbb{R}^{n})$, then 
\begin{equation*}
\Vert \varphi _{\varepsilon }\ast f\Vert _{{p(\cdot )}}\leq c\Vert \Psi
\Vert _{{1}}\Vert f\Vert _{{p(\cdot )}}
\end{equation*}%
for all $f\in L^{p(\cdot )}$, where 
\begin{equation*}
\varphi _{\varepsilon }:=\frac{1}{\varepsilon ^{n}}\varphi \left( \frac{%
\cdot }{\varepsilon }\right) , \quad \varepsilon >0.
\end{equation*}
We refer to \cite{CFMP} and \cite{Di}, where various results on maximal
functions on variable Lebesgue spaces are obtained.\vskip5pt

We put 
\begin{equation*}
\eta _{t,m}(x):=t^{-n}(1+t^{-1}\left\vert x\right\vert )^{-m}
\end{equation*}
for any $x\in \mathbb{R}^{n}$, $t>0$ and $m>0$. Note that $\eta _{t,m}\in
L^{1}$ when $m>n$ and that $\left\Vert \eta _{t,m}\right\Vert _{1}=c(m)$ is
independent of $t$. If $t=2^{-v}$, $v\in \mathbb{N}_{0}$ then we put 
\begin{equation*}
\eta _{2^{-v},m}:=\eta _{v,m}.
\end{equation*}
We refer to the recent monograph $\mathrm{\cite{CF13}}$ for further
properties, historical remarks and references on variable exponent spaces.

\subsection{Some technical lemmas}

In this subsection we present some useful results. The following lemma is
proved in \cite[Lemma 6.1]{DHR} (see also \cite[Lemma 19]{KV122}).

\begin{lem}
\label{DHR-lemma}Let $\alpha \in C_{\mathrm{loc}}^{\log }(\mathbb{R}^{n})$, $%
m\in \mathbb{N}_{0}$ and let $R\geq c_{\log }(\alpha )$, where $c_{\log
}(\alpha )$ is the constant from $\mathrm{\eqref{lo-log-Holder}}$ for $%
g=\alpha $. Then there exists a constant $c>0$ such that 
\begin{equation*}
t^{-\alpha (x)}\eta _{t,m+R}(x-y)\leq c\text{ }t^{-\alpha (y)}\eta
_{t,m}(x-y)
\end{equation*}%
for any $0<t\leq 1$ and $x,y\in \mathbb{R}^{n}$. In particular, if $\alpha $
decays logarithmically at the origin, then there exists a constant $c>0$
such that 
\begin{equation*}
t^{-\alpha (x)}\eta _{t,m+R}(x)\leq c\text{ }t^{-\alpha (0)}\eta _{t,m}(x)
\end{equation*}%
for any $0<t\leq 1$ and $x\in \mathbb{R}^{n}$.
\end{lem}

The previous lemma allows us to treat the variable smoothness in many cases
as if it were not variable at all. Namely, we can move the factor $%
t^{-\alpha (x)}$ inside the convolution as follows: 
\begin{equation*}
t^{-\alpha (x)}\eta _{t,m+R}\ast f(x)\leq c\text{ }\eta _{t,m}\ast
(t^{-\alpha (\cdot )}f)(x).
\end{equation*}

\begin{lem}
\label{r-trick}Let $r,N>0$, $m>n$ and $\theta ,\omega \in \mathcal{S}\left( 
\mathbb{R}^{n}\right) $ with $\mathrm{supp}\, \mathcal{F}\omega \subset 
\overline{B(0,1)}$. Then there exists a constant $c=c(r,m,n)>0$ such that
for all $g\in \mathcal{S}^{\prime }\left( \mathbb{R}^{n}\right) $, we have%
\begin{equation*}
\left\vert \theta _{N}\ast \omega _{N}\ast g\left( x\right) \right\vert \leq
c(\eta _{N,m}\ast \left\vert \omega _{N}\ast g\right\vert
^{r}(x))^{1/r},\quad x\in \mathbb{R}^{n},
\end{equation*}%
where $\theta _{N}(\cdot ):=N^{n}\theta (N\cdot )$, $\omega _{N}(\cdot
):=N^{n}\omega (N\cdot )$ and $\eta _{N,m}:=N^{n}(1+N\left\vert \cdot
\right\vert )^{-m}$.
\end{lem}

The proof of this lemma is given in \cite{D6} by using the same arguments of 
\cite[Chapter V, Theorem 5]{ST89}. \newline

The next three lemmas are proved in \cite{DHR}, where the first one tells us
that in most circumstances two convolutions are as good as one.

\begin{lem}
\label{Conv-est2} For $v_{0},v_{1}\in \mathbb{N}_{0}$ and $m>n$, we have%
\begin{equation*}
\eta _{v_{0},m}\ast \eta _{v_{1},m}\approx \eta _{\min (v_{0},v_{1}),m}
\end{equation*}%
with the constant depending only on $m$ and $n$.
\end{lem}

For $v\in \mathbb{N}_{0}$ and $m=(m_{1},...,m_{n})\in \mathbb{Z}^{n}$, let $%
Q_{v,m}$ be the dyadic cube in $\mathbb{R}^{n}$: 
\begin{equation*}
Q_{v,m}=\{(x_{1},...,x_{n}):m_{i}\leq 2^{v}x_{i}<m_{i}+1,i=1,2,...,n\}.
\end{equation*}%
As the collection of all such cubes, we define 
\begin{equation*}
\mathcal{Q}=\{Q_{v,m}:v\in \mathbb{N}_{0},m\in \mathbb{Z}^{n}\}.
\end{equation*}%
Then we have: %For each cube $Q$, we denote by
%$x_{Q_{v,m}}$ the lower left-corner $2^{-v}m$ of $Q=Q_{v,m}$, its
%side length by $l(Q)$.

\begin{lem}
\label{Conv-est1}Let $v\in \mathbb{N}_{0}$ and $m>n$. Then for any $Q\in 
\mathcal{Q}$ with $l(Q)=2^{-v}$, $y\in Q$ and $x\in \mathbb{R}^{n}$, we have%
\begin{equation*}
\eta _{v,m}\ast \left( \frac{\chi _{Q}}{|Q|}\right) (x)\approx \eta
_{v,m}(x-y)
\end{equation*}%
with the constant depending only on $m$ and $n$, where $l(Q)$ is the side
length of $Q$.
\end{lem}

The next lemma is the Hardy type inequality, which is easy to prove.

\begin{lem}
\label{Hardy-inequality} Let $0<a<1$, $\sigma \geq 0$ and $0<q\leq \infty$.
Let $\left\{ \varepsilon _{k}\right\} _{k}$be a sequence of positive real
numbers, and denote 
\begin{equation*}
\delta _{k}=\sum_{j=-\infty }^{\infty }\left\vert k-j\right\vert ^{\sigma
}a^{\left\vert k-j\right\vert }\varepsilon _{j}.
\end{equation*}
Then there exists a constant $c>0$ depending only on $a$ and $q$ such that%
\begin{equation*}
\Big(\sum\limits_{k=-\infty }^{\infty }\delta _{k}^{q}\Big)^{1/q}\leq c\text{
}\Big(\sum\limits_{k=-\infty }^{\infty }\varepsilon _{k}^{q}\Big)^{1/q}.
\end{equation*}
\end{lem}

Putting 
\begin{equation*}
w(Q):=\int_Q w(x)\, dx,
\end{equation*}
we have the following result (see Lemma 3.3 for $w=1$ from \cite{DHHMS}).

\begin{lem}
\label{DHHR-estimate}Let $p\in \mathcal{P}^{\mathrm{log}}(\mathbb{R}^{n})$,
and $w$ be a weight function on $\mathbb{R}^{n}$. Then, putting 
\begin{equation*}
\gamma _{m}:=e^{-4mc(1/p)}\in \left( 0,1\right) ,\quad c(1/p):=\max (c_{\log
}(1/p),c_{\log })
\end{equation*}%
for every $m>0$, and 
\begin{equation*}
p_{Q}^{-}:=\underset{z\in Q}{\mathrm{ess}\text{-}\mathrm{\inf }}\,p(z)
\end{equation*}%
for a cube $Q$ in $\mathbb{R}^{n}$, we have the inequality\textrm{:} 
\begin{equation}
\begin{split}
& \left( \frac{\gamma _{m}}{w(Q)}\int_{Q}|f(y)|w(y)\,dy\right) ^{p(x)} \\
& \leq c\max \left( 1,\left( w(Q)\right) ^{1-\frac{p\left( x\right) }{%
p_{Q}^{-}}}\right) \frac{1}{w(Q)}\int_{Q}\left\vert f(y)\right\vert
^{p(y)}w(y)\,dy \\
& +\frac{c\min (|Q|^{m},1)}{w(Q)}\int_{Q}\left\{
(e+|x|)^{-m}+(e+|y|)^{-m}\right\} w(y)\,dy
\end{split}
\label{EQ:Key}
\end{equation}%
for some positive\ constant $c>0$ every\ cube\ $($or ball$)$ $Q$, all $x\in Q
$ and all $f\in L^{p\left( \cdot \right) }(w)$\ with 
\begin{equation*}
\left\Vert f\right\Vert _{L^{p(\cdot )}(w)}\leq 1.
\end{equation*}%
If $Q=(a,b)\subset \mathbb{R}$ with $0<a<b<\infty $, then, putting 
\begin{equation*}
\gamma _{m}:=e^{-4mc_{\log }(1/p)}
\end{equation*}%
for every $m>0$, we have the following inequality\textrm{:}%
\begin{equation}
\begin{split}
& \left( \frac{\gamma _{m}}{w(Q)}\int_{Q}\left\vert f(y)\right\vert
w(y)\,dy\right) ^{p\left( x\right) } \\
\leq & c\text{ }\max \left( 1,\left( w(Q)\right) ^{1-\frac{p\left( x\right) 
}{p^{-}}}\right) \frac{1}{w(Q)}\int_{Q}\phi (y)w(y)\,dy \\
& +\frac{c\text{ }\omega (m,b)}{w(Q)}\int_{Q}g(x,y)w(y)\,dy
\end{split}
\label{EQ:key1}
\end{equation}%
for some positive\ constant $c>0$, all $x\in Q$ and all $f\in L^{p\left(
\cdot \right) }(w)$ with 
\begin{equation*}
\left\Vert f\right\Vert _{L^{p(\cdot )}(w)}\leq 1,
\end{equation*}%
where we put 
\begin{equation*}
\omega (m,b)=\min \left( b^{m},1\right) \text{,\quad }\phi (y)=\left\vert
f(y)\right\vert ^{p(y)}
\end{equation*}%
and%
\begin{equation*}
g(x,y)=\left( e+\frac{1}{x}\right) ^{-m}+\left( e+\frac{1}{y}\right) ^{-m},
\end{equation*}%
or%
\begin{equation*}
\omega (m,b)=\min \left( b^{m},1\right) \text{,\quad }\phi (y)=\left\vert
f(y)\right\vert ^{p(0)}
\end{equation*}%
and%
\begin{equation*}
g(x,y)=\left( e+\frac{1}{x}\right) ^{-m}\chi _{\{z\in Q:p(z){<}p(0)\}}(x)
\end{equation*}%
with $p\in \mathcal{P}(\mathbb{R})$ being $\log $-H\"{o}lder continuous at
the origin. In addition, we have the same estimate, when 
\begin{equation*}
\omega (m,b)=1\text{,\quad }\gamma _{m}=e^{-mc_{\log }},\quad \phi
(y)=\left\vert f(y)\right\vert ^{p_{\infty }}\quad 
\end{equation*}%
and%
\begin{equation*}
g(x,y)=(e+x)^{-m}\chi _{\{z\in Q:p(z){<}p_{\infty }\}}(x)
\end{equation*}%
with $p\in \mathcal{P}(\mathbb{R})$ satisfying the $\log $-H\"{o}lder decay
condition.
\end{lem}

%Notice that in the proof of Lemma \ref{DHHR-estimate} we need only that
%\begin{equation*}
%\int_{Q}\left\vert f(y)\right\vert ^{p\left( y\right) }w(y)\, dy\leq 1
%\end{equation*}%
%and/or $\left\Vert f\right\Vert _{\infty }\leq 1$.
The proof of Lemma \ref{DHHR-estimate} is postponed to appendix.

\begin{lem}
\label{Key-lemma1.1} Let $\left\{ f_{v}\right\} _{v\in \mathbb{N}}$ be a
sequence of measurable functions on $\mathbb{R}^{n}$ and $p\in \mathcal{P}(%
\mathbb{R}^{n})$. Let $q\in \mathcal{P}(\mathbb{R})$ be $\log$-H\"{o}lder
continuous at the origin. Then 
\begin{equation*}
\Big\|\Big(t^{-\frac{1}{q(t)}}\left\Vert f_{v}\right\Vert _{p(\cdot )}\chi
_{\lbrack 2^{-v},2^{1-v}]}\Big)_{v}\Big\|_{\ell _{>}^{q(\cdot )}(L^{q(\cdot
)})}\approx \left(\sum_{v=1}^{\infty }\left\Vert f_{v}\right\Vert _{p(\cdot
)}^{q(0)}\right)^{\frac{1}{q(0)}},
\end{equation*}%
where the implicit positive constants are independent of $f_{v},v\in \mathbb{%
N}$.
\end{lem}

\begin{proof} We divide the proof into two steps. \\

\textit{Step 1.} Let $\left\{ f_{v}\right\} _{v\in \mathbb{N}}$ be a
sequence of measurable functions such that the right-hand side is
less than or equal to 1. We will prove that
\begin{equation*}
\sum_{v=1}^{\infty }\int_{2^{-v}}^{2^{1-v}}\left\Vert
f_{v}\right\Vert _{p(\cdot )}^{q(t)}\frac{dt}{t}\lesssim 1.
\end{equation*}%
Our estimate clearly follows from the inequality%
\begin{equation*}
\int_{2^{-v}}^{2^{1-v}}\left\Vert f_{v}\right\Vert _{p(\cdot )}^{q(t)}\frac{%
dt}{t}\lesssim \left\Vert f_{v}\right\Vert _{p(\cdot
)}^{q(0)}+2^{-v}=:\delta_v
\end{equation*}%
for any $v\in \mathbb{N}$. This claim can be reformulated as showing\ that%
\begin{equation*}
\int_{2^{-v}}^{2^{1-v}}\big(\delta_v^{-\frac{1}{q(t)}}\left\Vert
f_{v}\right\Vert _{p(\cdot )}\big)^{q(t)}\frac{dt}{t}\lesssim 1.
\end{equation*}%
Let us prove that%
\begin{equation*}
\delta_v^{-\frac{1}{q(t)}}\left\Vert f_{v}\right\Vert _{p(\cdot
)}\lesssim 1
\end{equation*}%
for any $t\in \lbrack 2^{-v},2^{1-v}]$ and any $v\in \mathbb{N}$. We
use the
log-H\"{o}lder continuity of $q$ at the origin to show that%
\begin{equation*}
\delta_v^{-\frac{q\left( 0\right) }{q(t)}}\approx \delta ^{-1},\text{ \ \ }%
t\in \lbrack 2^{-v},2^{1-v}],v\in \mathbb{N}.
\end{equation*}%
Therefore, from the definition of $\delta_v$, we find that%
\begin{equation*}
\delta_v^{-\frac{q\left( 0\right) }{q(t)}}\left\Vert
f_{v}\right\Vert _{p(\cdot )}^{q\left( 0\right) }\lesssim 1.
\end{equation*}%
Hence we obtain the desired estimate.\\

\textit{Step 2.} Let $\left\{ f_{v}\right\} _{v\in \mathbb{N}}$ be a
sequence of measurable functions such that the left-hand side is
less than or equal to 1. We will prove that
\begin{equation*}
\sum_{v=1}^{\infty }\left\Vert f_{v}\right\Vert _{p(\cdot
)}^{q(0)}\lesssim 1.
\end{equation*}%
Clearly, the estimate follows from the inequality:
\begin{equation*}
\left\Vert f_{v}\right\Vert _{p(\cdot )}^{q(0)}\lesssim
\int_{2^{-v}}^{2^{1-v}}\left\Vert f_{v}\right\Vert _{p(\cdot )}^{q(\tau )}%
\frac{d\tau }{\tau }+2^{-v}=:\delta_v
\end{equation*}%
for any $v\in \mathbb{N}$. This claim can be reformulated as showing that%
\begin{equation*}
\big(\delta_v^{-\frac{1}{q(0)}}\left\Vert f_{v}\right\Vert _{p(\cdot )}\big)%
^{q(0)}=\Big(\frac{1}{\log 2}\int_{2^{-v}}^{2^{1-v}}\delta_v^{-\frac{1}{q(0)}%
}\left\Vert f_{v}\right\Vert _{p(\cdot )}\frac{d\tau }{\tau }\Big)%
^{q(0)}\lesssim 1.
\end{equation*}%
By Lemma \ref{DHHR-estimate}, we have
\begin{equation*}
\Big(\frac{\gamma_m}{\log 2}\int_{2^{-v}}^{2^{1-v}}\delta_v^{-\frac{1}{q(0)}%
}\left\Vert f_{v}\right\Vert _{p(\cdot )}\frac{d\tau }{\tau
}\Big)^{q\left( t\right) }\lesssim \int_{2^{-v}}^{2^{1-v}}\delta_v
^{-\frac{q\left( \tau \right) }{q(0)}}\left\Vert f_{v}\right\Vert
_{p(\cdot )}^{q\left( \tau \right) }\frac{d\tau }{\tau }+1,
\end{equation*}%
where $\gamma_m=e^{-2mc_{\log }(1/p)}$ and $m>0$. We use the
log-H\"{o}lder
continuity of $q$ at the origin to show that%
\begin{equation*}
\delta_v^{-\frac{q\left( \tau \right) }{q(0)}}\approx \delta_v^{-1},
\text{ \ \ }\tau \in \lbrack 2^{-v},2^{1-v}],v\in \mathbb{N}.
\end{equation*}%
Therefore, from the definition of $\delta $, we find that%
\begin{equation*}
\int_{2^{-v}}^{2^{1-v}}\delta_v^{-1}\left\Vert f_{v}\right\Vert
_{p(\cdot )}^{q\left( \tau \right) }\frac{d\tau }{\tau }\lesssim 1
\end{equation*}%
for any $v\in \mathbb{N}$, which implies that%
\begin{equation*}
\big(\delta_v^{-\frac{1}{q(0)}}\left\Vert f_{v}\right\Vert _{p(\cdot )}\big)%
^{q(0)}\lesssim 1
\end{equation*}%
for any $v\in \mathbb{N}$. The proof of Lemma \ref{Key-lemma1.1} is
complete.
\end{proof}

By a similar argument as in Lemma \ref{Key-lemma1.1},\ we have:

\begin{lem}
\label{Key-lemma1.2}Let $\left\{ f_{v}\right\} _{v\in \mathbb{N}}$ be a
sequence of measurable functions on $\mathbb{R}^{n}$, $p\in \mathcal{P}(%
\mathbb{R}^{n})$ and $\alpha \in C_{\mathrm{loc}}^{\log }\left( 
%TCIMACRO{\U{211d} }%
%BeginExpansion
\mathbb{R}
%EndExpansion
^{n}\right) $\textit{.\ }Let $q\in \mathcal{P}(\mathbb{R})$\ \textit{be log-H%
\"{o}lder continuous at the origin}. Then 
\begin{equation*}
\Big\|\Big(\big\|t^{-\alpha (\cdot )-\frac{1}{q(t)}}f_{v}\chi _{Q_{v}}\big\|%
_{p(\cdot )}\chi _{\lbrack 2^{-v},2^{1-v}]}\Big)_{v}\Big\|_{\ell
_{>}^{q(\cdot )}(L^{q(\cdot )})}\approx \left(\sum_{v=1}^{\infty }\big\|%
2^{v\alpha (\cdot )}f_{v}\chi _{Q_{v}}\big\|_{p(\cdot )}^{q(0)}\right)^{%
\frac{1}{q(0)}},
\end{equation*}%
where $Q_{v}$, $v\in \mathbb{N}$, are dyadic cubes on $\mathbb{R}^{n}$\ with 
$\ell (Q_{v})=2^{-v}$, and the implicit positive constants are independent
of $f_{v},v\in \mathbb{N}$.
\end{lem}

The next two lemmas are the continuous version of the Hardy type inequality,
where the second lemma for constant exponents is from \cite{Mo87}.

\begin{lem}
\label{Key-lemma1}Let $p\in \mathcal{P}(\mathbb{R}^{n})$. Let $q\in \mathcal{%
P}(\mathbb{R})$ be $\log$-H\"{o}lder continuous at the origin with $1\leq
q^{-}\leq q^{+}<\infty $. Let $\left\{ f_{v}\right\} _{v\in \mathbb{N}} $ be
a sequence of measurable functions on $\mathbb{R}^{n}$. For all $x\in 
\mathbb{R}^{n}$, $v\in \mathbb{N}$ and all $\delta >0$, let 
\begin{equation*}
g_{v}(x)=\sum_{k=1}^{\infty }2^{-|k-v|\delta }f_{k}(x).
\end{equation*}
Then there exists a positive constant $c$, independent of $\left\{
f_{v}\right\} _{v\in \mathbb{N}}$ such that%
\begin{equation*}
\Big\|\Big(t^{-\frac{1}{q(t)}}\left\Vert g_{v}\right\Vert _{p(\cdot )}\chi
_{\lbrack 2^{-v},2^{1-v}]}\Big)_{v}\Big\|_{\ell _{>}^{q(\cdot )}(L^{q(\cdot
)})}\leq c\Big\|\Big(t^{-\frac{1}{q(t)}}\left\Vert f_{v}\right\Vert
_{p(\cdot )}\chi _{\lbrack 2^{-v},2^{1-v}]}\Big)_{v}\Big\|_{\ell
_{>}^{q(\cdot )}(L^{q(\cdot )})}.
\end{equation*}
\end{lem}

\begin{proof} The required estimate follows from Lemmas \ref
{Hardy-inequality} and \ref{Key-lemma1.1}.
\end{proof}

\begin{lem}
\label{lq-inequality} Let $s>0$, and let $q\in \mathcal{P}(\mathbb{R})$ be $%
\log $-H\"{o}lder continuous at the origin with $1\leq q^{-}\leq
q^{+}<\infty $. Let $\{ \varepsilon _{t}\} _{t}$ be a sequence of positive
measurable functions. Let%
\begin{equation*}
\eta _{t}=t^{s}\int_{t}^{1}\tau ^{-s}\varepsilon _{\tau }\frac{d\tau }{\tau }%
\quad \text{and\quad }\delta _{t}=t^{-s}\int_{0}^{t}\tau ^{s}\varepsilon
_{\tau }\frac{d\tau }{\tau }.
\end{equation*}%
Then there exists a constant $c>0\ $\textit{depending only on }$s$, $q^{-}$%
\textit{, c}$_{\log }(q)$ \textit{and }$q^{+}$ such that%
\begin{equation*}
\left\Vert \eta _{t}\right\Vert _{L^{q(\cdot )}((0,1],\frac{dt}{t}%
)}+\left\Vert \delta _{t}\right\Vert _{L^{q(\cdot )}((0,1],\frac{dt}{t}%
)}\leq c\left\Vert \varepsilon _{t}\right\Vert _{L^{q(\cdot )}((0,1],\frac{dt%
}{t})}.
\end{equation*}
\end{lem}

\begin{proof} This lemma is proved in \cite{D.S2007}, and here, we
give an alternative proof. We suppose that $\left\Vert \varepsilon
_{t}\right\Vert
_{L^{q(\cdot )}((0,1],\frac{dt}{t})}\leq 1$. Notice that%
\begin{equation*}
\left\Vert \eta _{t}\right\Vert _{L^{q(\cdot )}((0,1],\frac{dt}{t})}\approx %
\Big\|\Big(t^{-\frac{1}{q(t)}}\eta _{t}\chi _{\lbrack 2^{-v},2^{1-v}]}\Big)%
_{v}\Big\|_{\ell _{>}^{q(\cdot )}(L^{q(\cdot )})}.
\end{equation*}%
We see that%
\begin{equation*}
\int_{2^{-v}}^{1}\tau ^{-s}\varepsilon _{\tau }\frac{d\tau }{\tau }%
=\sum_{i=0}^{v-1}\int_{2^{i-v}}^{2^{i-v+1}}\tau ^{-s}\varepsilon _{\tau }%
\frac{d\tau }{\tau }=\sum_{j=1}^{v}\int_{2^{-j}}^{2^{1-j}}\tau
^{-s}\varepsilon _{\tau }\frac{d\tau }{\tau }.
\end{equation*}%
Let $\sigma >0$ be such that $q^{+}<\sigma $. We have%
\begin{eqnarray*}
\Big(\sum_{j=1}^{v}\int_{2^{-j}}^{2^{1-j}}\tau ^{-s}\varepsilon _{\tau }%
\frac{d\tau }{\tau }\Big)^{q(t)/\sigma } &\leq &\sum_{j=1}^{v}\Big(%
\int_{2^{-j}}^{2^{1-j}}\tau ^{-s}\varepsilon _{\tau }\frac{d\tau }{\tau }%
\Big)^{q(t)/\sigma } \\
&\leq &\sum_{j=1}^{v}2^{jsq(t)/\sigma }\Big(\int_{2^{-j}}^{2^{1-j}}%
\varepsilon _{\tau }\frac{d\tau }{\tau }\Big)^{q(t)/\sigma } \\
&=&2^{vsq(t)/\sigma }\sum_{j=1}^{v}2^{(j-v)sq(t)/\sigma }\Big(%
\int_{2^{-j}}^{2^{1-j}}\varepsilon _{\tau }\frac{d\tau }{\tau }\Big)%
^{q(t)/\sigma }.
\end{eqnarray*}%
By H\"{o}lder's inequality, we estimate this expression by%
\begin{equation*}
2^{vsq(t)/\sigma }\Big(\sum_{j=1}^{v}2^{(j-v)sq(t)/\sigma }\Big(%
\int_{2^{-j}}^{2^{1-j}}\varepsilon _{\tau }\frac{d\tau }{\tau }\Big)^{q(t)}%
\Big)^{1/\sigma }.
\end{equation*}%
By Lemma \ref{DHHR-estimate} we find an $m>0$ such that%
\begin{equation*}
\Big(\frac{1}{(v-j+1)\log 2}\int_{2^{-v}}^{2^{1-j}}\varepsilon
_{\tau }\chi _{\lbrack 2^{-j},2^{1-j}]}(\tau )\frac{d\tau }{\tau
}\Big)^{q(t)}\lesssim
\frac{1}{v-j+1}\int_{2^{-j}}^{2^{1-j}}\varepsilon _{\tau }^{q(\tau )}\frac{%
d\tau }{\tau }+2^{-jm}
\end{equation*}%
for any $v\geq j$ and any $t\in \lbrack 2^{-v},2^{1-v}]\subset
\lbrack 2^{-v},2^{1-j}]$. Therefore, we get
\begin{equation*}
\eta _{t}^{q(t)}\lesssim \sum_{j=1}^{v}2^{(j-v)sq^{-}/\sigma
}(v-j+1)^{q^{+}-1}\int_{2^{-j}}^{2^{1-j}}\varepsilon _{\tau }^{q(\tau )}%
\frac{d\tau }{\tau }+h_{v}
\end{equation*}%
for any $t\in \lbrack 2^{-v},2^{1-v}]$, where%
\begin{equation*}
h_{v}=\sum_{j=1}^{v}2^{(j-v)sq^{-}/\sigma
}(v-j+1)^{q^{+}}2^{-jm},\quad v\in \mathbb{N}.
\end{equation*}%
Observe that $$\int_{2^{-v}}^{2^{1-v}}\frac{dt}{t}=\log 2.$$ Then
\begin{equation*}
\int_{2^{-v}}^{2^{1-v}}\eta _{t}^{q(t)}\frac{dt}{t}\lesssim
\sum_{j=1}^{v}2^{(j-v)sq^{-}/\sigma
}(v-j+1)^{q^{+}-1}\int_{2^{-j}}^{2^{1-j}}\varepsilon _{\tau }^{q(\tau )}%
\frac{d\tau }{\tau }+h_{v}.
\end{equation*}%
Applying Lemma \ref{Hardy-inequality}, we get%
\begin{equation*}
\sum_{v=1}^{\infty }\int_{2^{-v}}^{2^{1-v}}\eta _{t}^{q(t)}\frac{dt}{t}%
\lesssim \sum_{j=1}^{\infty }\int_{2^{-j}}^{2^{1-j}}\varepsilon
_{\tau }^{q(\tau )}\frac{d\tau }{\tau }+c\lesssim 1
\end{equation*}%
by taking $m$ large enough such that $m>0$. Hence, we get
\begin{equation*}
\left\Vert \eta _{t}\right\Vert _{L^{q(\cdot
)}((0,1],\frac{dt}{t})}\lesssim 1\text{.}
\end{equation*}%
Now we prove that%
\begin{equation*}
\left\Vert \delta _{t}\right\Vert _{L^{q(\cdot )}((0,1],\frac{dt}{t}%
)}\lesssim 1.
\end{equation*}%
Notice that%
\begin{equation*}
\left\Vert \delta _{t}\right\Vert _{L^{q(\cdot )}((0,1],\frac{dt}{t}%
)}\approx \Big\|\Big(t^{-\frac{1}{q(t)}}\delta _{t}\chi _{\lbrack
2^{-v},2^{1-v}]}\Big)_{v}\Big\|_{\ell _{>}^{q(\cdot )}(L^{q(\cdot
)})}.
\end{equation*}%
We have%
\begin{equation*}
\int_{0}^{2^{1-v}}\tau ^{s}\varepsilon _{\tau }\frac{d\tau }{\tau }%
=\sum_{i=-\infty }^{-v}\int_{2^{i}}^{2^{i+1}}\tau ^{s}\varepsilon _{\tau }%
\frac{d\tau }{\tau }=\sum_{j=v}^{\infty }\int_{2^{-j}}^{2^{1-j}}\tau
^{s}\varepsilon _{\tau }\frac{d\tau }{\tau }.
\end{equation*}%
Let $\sigma >0$ be such that $q^{+}<\sigma $. We have%
\begin{eqnarray*}
\Big(\sum_{j=v}^{\infty }\int_{2^{-j}}^{2^{1-j}}\tau ^{s}\varepsilon _{\tau }%
\frac{d\tau }{\tau }\Big)^{q(t)/\sigma } &\leq &\sum_{j=v}^{\infty }\Big(%
\int_{2^{-j}}^{2^{1-j}}\tau ^{s}\varepsilon _{\tau }\frac{d\tau }{\tau }\Big)%
^{q(t)/\sigma } \\
&\leq &\sum_{j=v}^{\infty }2^{-jsq(t)/\sigma }\Big(\int_{2^{-j}}^{2^{1-j}}%
\varepsilon _{\tau }\frac{d\tau }{\tau }\Big)^{q(t)/\sigma } \\
&=&2^{-vsq(t)/\sigma }\sum_{j=v}^{\infty }2^{(v-j)sq(t)/\sigma }\Big(%
\int_{2^{-j}}^{2^{1-j}}\varepsilon _{\tau }\frac{d\tau }{\tau }\Big)%
^{q(t)/\sigma }.
\end{eqnarray*}%
Again, by H\"{o}lder's inequality, we estimate this expression by%
\begin{equation*}
2^{-vsq(t)/\sigma }\Big(\sum_{j=v}^{\infty }2^{(v-j)sq(t)/\sigma }\Big(%
\int_{2^{-j}}^{2^{1-j}}\varepsilon _{\tau }\frac{d\tau }{\tau }\Big)^{q(t)}%
\Big)^{1/\sigma }.
\end{equation*}%
Applying again Lemmas \ref{DHR-lemma} and \ref{DHHR-estimate}, we get%
\begin{equation*}
\Big(\frac{1}{(j-v+1)\log 2}\int_{2^{-j}}^{2^{1-v}}\varepsilon
_{\tau }\chi _{\lbrack 2^{-j},2^{1-j}]}(\tau )\frac{d\tau }{\tau
}\Big)^{q(t)}\lesssim
\frac{1}{j-v+1}\int_{2^{-j}}^{2^{1-j}}\varepsilon _{\tau }^{q(\tau )}\frac{%
d\tau }{\tau }+2^{-vm}
\end{equation*}%
for any $j\geq v$ and any $t\in \lbrack 2^{-v},2^{1-v}]\subset
\lbrack 2^{-j},2^{1-v}]$. Therefore, we deduce that
\begin{equation*}
\delta _{t}^{q(t)}\lesssim \sum_{j=v}^{\infty }2^{(v-j)sq(t)/\sigma
}(j-v+1)^{q^{+}-1}\int_{2^{-j}}^{2^{1-j}}\varepsilon _{\tau }^{q(\tau )}%
\frac{d\tau }{\tau }+f_{v}
\end{equation*}%
for any $t\in \lbrack 2^{-v},2^{1-v}]$, where%
\begin{equation*}
f_{v}=2^{-vm},\quad v\in \mathbb{N}.
\end{equation*}%
Again, observe that $$\int_{2^{-v}}^{2^{1-v}}\frac{dt}{t}=\log 2.$$
Then
\begin{equation*}
\int_{2^{-v}}^{2^{1-v}}\delta _{t}^{q(t)}\frac{dt}{t}\lesssim
\sum_{j=v}^{\infty }2^{(v-j)sq(t)/\sigma
}(j-v+1)^{q^{+}-1}\int_{2^{-j}}^{2^{1-j}}\varepsilon _{\tau }^{q(\tau )}%
\frac{d\tau }{\tau }+f_{v}.
\end{equation*}%
By taking $m$ large enough such that $m>0$ and again\ by Lemma \ref%
{Hardy-inequality}\ we\ get%
\begin{equation*}
\sum_{v=1}^{\infty }\int_{2^{-v}}^{2^{1-v}}\delta _{t}^{q(t)}\frac{dt}{t}%
\lesssim \sum_{j=1}^{\infty }\int_{2^{-j}}^{2^{1-j}}\varepsilon
_{\tau }^{q(\tau )}\frac{d\tau }{\tau }+c\lesssim 1.
\end{equation*}%
The proof of Lemma \ref{lq-inequality} is completed by the scaling
argument.
\end{proof}

The following lemma is from \cite[Lemma 1]{Ry01}.

\begin{lem}
\label{Ry-Lemma1} Let $\rho ,\mu \in \mathcal{S}(\mathbb{R}^{n})$, and $%
M\geq -1$ an integer such that 
\begin{equation*}
\int_{\mathbb{R}^{n}}x^{\alpha }\mu (x)dx=0
\end{equation*}
for all $\left\vert \alpha \right\vert \leq M$. Then for any $N>0$, there is
a constant $c(N)>0$ such that 
\begin{equation*}
\sup_{z\in \mathbb{R}^{n}}\left\vert t^{-n}\mu (t^{-1}\cdot )\ast \rho
(z)\right\vert (1+\left\vert z\right\vert )^{N}\leq c(N)\text{ }t^{M+1}.
\end{equation*}
\end{lem}

\section{Variable\ Besov\ spaces}

In this section we\ present the definition of Besov spaces of variable
smoothness and integrability, and prove the basic properties in analogy to
the Besov spaces with fixed exponents. Select a pair of Schwartz functions $%
\Phi $ and $\varphi $ satisfying 
\begin{equation}
\mathrm{supp}\, \mathcal{F}\Phi \subset \{x\in \mathbb{R}^{n}:\left\vert
x\right\vert <2\}, \quad \mathrm{supp}\, \mathcal{F}\varphi \subset
\left\{x\in \mathbb{R}^{n}:\frac{1}{2}<\left\vert x\right\vert <2\right\}
\label{Ass1}
\end{equation}%
and 
\begin{equation}
\mathcal{F}\Phi (\xi )+\int_{0}^{1}\mathcal{F}\varphi (t\xi )\frac{dt}{t}%
=1,\quad \xi \in \mathbb{R}^{n}.  \label{Ass2}
\end{equation}%
Such a resolution \eqref{Ass1} and \eqref{Ass2} of unity can be constructed
as follows. Let $\mu \in \mathcal{S}(\mathbb{R}^{n})$\ be such that $%
\left\vert \mathcal{F}\mu (\xi )\right\vert >0$\ for $\frac{1}{2}<\left\vert
\xi \right\vert <2$. There exists $\eta \in \mathcal{S}(\mathbb{R}^{n})$\
with 
\begin{equation*}
\mathrm{supp} \, \mathcal{F}\eta \subset \left\{x\in \mathbb{R}^{n}:\frac{1}{%
2}<\left\vert x\right\vert <2\right\}
\end{equation*}
such that 
\begin{equation*}
\int_{0}^{\infty }\mathcal{F}\mu (t\xi )\, \mathcal{F}\eta (t\xi )\frac{dt}{t%
}=1,\quad \xi \neq 0
\end{equation*}%
(see \cite{CaTor75}, \cite{Hei74} and \cite{JaTa81}). We set $\mathcal{F}%
\varphi =\mathcal{F}\mu \, \mathcal{F}\eta $ and 
\begin{equation*}
\mathcal{F}\Phi (\xi )=\left\{ 
\begin{array}{ccc}
\displaystyle{\int_{1}^{\infty}} \mathcal{F}\varphi (t\xi )\, \dfrac{dt}{t}
& \text{if} & \xi \neq 0, \\ 
1 & \text{if} & \xi =0.%
\end{array}%
\right.
\end{equation*}%
Then $\mathcal{F}\Phi \in \mathcal{S}(\mathbb{R}^{n})$, and as $\mathcal{F}%
\eta $ is supported in $\{x\in \mathbb{R}^{n}:\frac{1}{2}<\left\vert
x\right\vert <2\}$, we see that supp $\mathcal{F}\Phi \subset \{x\in \mathbb{%
R}^{n}:\left\vert x\right\vert <2\}$.\newline

Now we define the spaces under consideration.

\begin{defn}
\label{B-F-def}Let $\alpha :\mathbb{R}^{n}\rightarrow \mathbb{R}$, $p\in 
\mathcal{P}(\mathbb{R}^{n})$ and $q\in \mathcal{P}(\mathbb{R})$. Let $\{%
\mathcal{F}\Phi ,\mathcal{F}\varphi \}$ be a resolution of unity and we put $%
\varphi _{t}=t^{-n}\varphi (\frac{\cdot }{t})$. The Besov space $\boldsymbol{%
B}_{p(\cdot ),q(\cdot )}^{\alpha (\cdot )}$\ is the collection of all $f\in 
\mathcal{S}^{\prime }(\mathbb{R}^{n})$\ such that 
\begin{equation}  \label{B-def}
\left\Vert f\right\Vert _{\boldsymbol{B}_{p(\cdot ),q(\cdot )}^{\alpha
(\cdot )}}:=\left\Vert \Phi \ast f\right\Vert _{p(\cdot )}+\left\Vert
\left\Vert t^{-\alpha (\cdot )}(\varphi _{t}\ast f)\right\Vert _{p(\cdot
)}\right\Vert _{L^{q(\cdot )}((0,1],\frac{dt}{t})}<\infty .
\end{equation}
\end{defn}

When $q=\infty, $\ the Besov space $\boldsymbol{B}_{p(\cdot ),\infty
}^{\alpha (\cdot )}$\ consist of all distributions $f\in \mathcal{S}^{\prime
}(\mathbb{R}^{n})$\ such that 
\begin{equation*}
\left\Vert \Phi \ast f\right\Vert _{p(\cdot )}+\sup_{t\in (0,1]}\left\Vert
t^{-\alpha (\cdot )}(\varphi _{t}\ast f)\right\Vert _{p(\cdot )}<\infty .
\end{equation*}

\vspace{5mm}

One recognizes immediately that $\boldsymbol{B}_{p(\cdot ),q(\cdot
)}^{\alpha (\cdot )}$ is a normed space, and if $\alpha $, $p$ and $q$ are
constants, then $\boldsymbol{B}_{p,q}^{\alpha }$ is the usual Besov spaces.
For general literature on function spaces of variable smoothness and
integrability we refer to [2-4, 10-11, 13-15, 18, 23, 25-26, 28, 32, 40-45].%
\newline

Now, we are ready to show that the definition of these\ function spaces is
independent of the chosen resolution $\{\mathcal{F}\Phi ,\mathcal{F}\varphi
\}$ of unity. This justifies our omission of the subscript $\Phi $ and $%
\varphi $ in the sequel.

\begin{thm}
\label{14} Let $\{\mathcal{F}\Phi ,\mathcal{F}\varphi \}$ and $\left\{ 
\mathcal{F}\Psi ,\mathcal{F}\psi \right\}$ be two resolutions of unity, and $%
p\in \mathcal{P}^{\log }\left( \mathbb{R}^{n}\right) $ and $\alpha \in C_{%
\mathrm{loc}}^{\log }\left( \mathbb{R}^{n}\right) $. Let $q\in \mathcal{P}(%
\mathbb{R})$ be log-H\"older continuous at the origin. Then 
\begin{equation*}
\left\Vert f\right\Vert _{\boldsymbol{B}_{p(\cdot ),q(\cdot )}^{\alpha
(\cdot )}}^{\Phi ,\varphi }\approx \left\Vert f\right\Vert _{\boldsymbol{B}%
_{p(\cdot ),q(\cdot )}^{\alpha (\cdot )}}^{\Psi ,\psi }.
\end{equation*}
\end{thm}

\begin{proof} It is sufficient to show that there exists a constant
$c>0$ such that for all $%
f\in \boldsymbol{B}_{p(\cdot ),q(\cdot )}^{\alpha (\cdot )}$ we have $$%
\left\Vert f\right\Vert _{\boldsymbol{B}_{p(\cdot ),q(\cdot
)}^{\alpha
(\cdot )}}^{\Phi ,\varphi }\leq c\left\Vert f\right\Vert _{\boldsymbol{B}%
_{p(\cdot ),q(\cdot )}^{\alpha (\cdot )}}^{\Psi ,\psi }.$$
Interchanging the roles of $\left( \Psi ,\psi \right) $ and $\left(
\Phi ,\varphi \right) $ we
obtain the desired result. We have%
\begin{equation*}
\mathcal{F}\Phi (\xi )=\mathcal{F}\Phi (\xi )\mathcal{F}\Psi (\xi
)+\int_{1/4}^{1}\mathcal{F}\Phi (\xi )\mathcal{F}\psi (\tau \xi
)\frac{d\tau }{\tau }
\end{equation*}%
and%
\begin{equation*}
\mathcal{F}\varphi (t\xi )=\int_{t/4}^{\min (1,4t)}\mathcal{F}\varphi (t\xi )%
\mathcal{F}\psi (\tau \xi )\frac{d\tau }{\tau }+\left\{
\begin{array}{ccc}
0, & \text{if} & 0<t<\frac{1}{4}; \\
\mathcal{F}\varphi (t\xi )\mathcal{F}\Psi (\xi ), & \text{if} & \frac{1}{4}%
\leq t\leq 1%
\end{array}%
\right.
\end{equation*}%
for any $\xi \in
%TCIMACRO{\U{211d} }%
%BeginExpansion
\mathbb{R}
%EndExpansion
^{n}$. Then we see that
\begin{equation}
\Phi \ast f=\Phi \ast \Psi \ast f+\int_{1/4}^{1}\Phi \ast \psi _{\tau }\ast f%
\frac{d\tau }{\tau } \label{dec-phi}
\end{equation}%
and%
\begin{equation*}
\varphi _{t}\ast f=\int_{t/4}^{\min (1,4t)}\varphi _{t}\ast \psi
_{\tau }\ast f\frac{d\tau }{\tau }+\left\{
\begin{array}{ccc}
0, & \text{if} & 0<t<\frac{1}{4}; \\
\varphi _{t}\ast \Psi \ast f, & \text{if} & \frac{1}{4}\leq t\leq 1.%
\end{array}%
\right.
\end{equation*}%
Since\ $p\in \mathcal{P}^{\log }\left(
%TCIMACRO{\U{211d} }%
%BeginExpansion
\mathbb{R}
%EndExpansion
^{n}\right) $, the convolution with a radially decreasing $L^{1}$%
-function is bounded on $L^{p(\cdot )}$:%
\begin{equation*}
\left\Vert \Phi \ast f\right\Vert _{p(\cdot )}\lesssim \left\Vert
\Psi \ast f\right\Vert _{p(\cdot )}+\int_{1/4}^{1}\left\Vert \psi
_{\tau }\ast f\right\Vert _{p(\cdot )}\frac{d\tau }{\tau }\lesssim
\left\Vert f\right\Vert _{\boldsymbol{B}_{p(\cdot ),q(\cdot
)}^{\alpha (\cdot )}}^{\Psi ,\psi },
\end{equation*}%
and%
\begin{eqnarray*}
&&\left\Vert t^{-\alpha (\cdot )-1/q(t)}(\varphi _{t}\ast
f)\right\Vert
_{p(\cdot )} \\
&\lesssim &\int_{t/4}^{\min (1,4t)}\left\Vert \tau ^{-\alpha (\cdot
)-1/q(\tau )}(\psi _{\tau }\ast f)\right\Vert _{p(\cdot
)}\frac{d\tau }{\tau }+\left\{
\begin{array}{ccc}
0, & \text{if} & 0<t<\frac{1}{4}; \\
\left\Vert \varphi _{t}\ast \Psi \ast f\right\Vert _{p(\cdot )}, &
\text{if}
& \frac{1}{4}\leq t\leq 1,%
\end{array}%
\right.
\end{eqnarray*}%
where we used%
\begin{equation*}
t^{-\alpha (x)-1/q(t)}\approx 1,\quad x\in \mathbb{R}^{n}.
\end{equation*}%
If $\frac{1}{4}\leq t\leq 1$, then
\begin{equation*}
t^{-1/q(t)}=\Big(\frac{t}{\tau }\Big)^{-1/q(t)}\tau ^{-1/q(t)}\leq
\tau ^{-1/q(t)}\lesssim \tau ^{-1/q(\tau )},\text{ \ \ \
}\frac{t}{4}\leq \tau \leq \min (1,4t),
\end{equation*}%
and%
\begin{equation*}
t^{-\alpha (x)}\lesssim (1+t^{-1}|x-y|)^{c_{\log }(\alpha )}\tau
^{-\alpha (y)},\quad x,y\in \mathbb{R}^{n}
\end{equation*}%
by Lemma \ref{DHR-lemma} and again the fact that the convolution
with a radially decreasing $L^{1}$-function is bounded on
$L^{p(\cdot )}$. H\"{o}lder's inequality gives that for any
$\frac{1}{4}\leq t\leq 1$
\begin{equation*}
\left\Vert t^{-\alpha (\cdot )-\frac{1}{q(t)}}(\varphi _{t}\ast
f)\right\Vert _{p(\cdot )}\lesssim \left\Vert f\right\Vert
_{\boldsymbol{B}_{p(\cdot ),q(\cdot )}^{\alpha (\cdot )}}^{\Psi
,\psi }\left\Vert \chi _{\lbrack t/4,\min (1,4t)]}\right\Vert
_{q^{\prime }(\cdot )}+\left\Vert \Psi \ast
f\right\Vert _{p(\cdot )}\lesssim \left\Vert f\right\Vert _{\boldsymbol{B}%
_{p(\cdot ),q(\cdot )}^{\alpha (\cdot )}}^{\Psi ,\psi }.
\end{equation*}%
Taking the $L^{q(\cdot )}([\frac{1}{4},1])$-norm we obtain the
desired estimate. In case\ $0<t<\frac{1}{4}$ we obtain
\begin{equation*}
\left\Vert t^{-\alpha (\cdot )}(\varphi _{t}\ast f)\right\Vert
_{p(\cdot
)}\lesssim \int_{t/4}^{4t}\min ((\frac{\tau }{t})^{s},(\frac{t}{\tau }%
)^{s})\left\Vert \tau ^{-\alpha (\cdot )}(\psi _{\tau }\ast
f)\right\Vert _{p(\cdot )}\frac{d\tau }{\tau },\quad s>0.
\end{equation*}%
Taking again the $L^{q(\cdot )}((0,\frac{1}{4}])$-norm and using Lemma \ref%
{lq-inequality} we conclude that
\begin{eqnarray*}
\left\Vert \left\Vert t^{-\alpha (\cdot )}(\varphi _{t}\ast
f)\right\Vert _{p(\cdot )}\right\Vert _{L^{q(\cdot
)}((0,\frac{1}{4}],\frac{dt}{t})} &\lesssim &\left\Vert \left\Vert
\tau ^{-\alpha (\cdot )}(\varphi _{t}\ast
f)\right\Vert _{p(\cdot )}\right\Vert _{L^{q(\cdot )}((0,\frac{1}{4}],\frac{%
d\tau }{\tau })} \\
&\lesssim &\left\Vert f\right\Vert _{\boldsymbol{B}_{p(\cdot
),q(\cdot )}^{\alpha (\cdot )}}^{\Psi ,\psi }.
\end{eqnarray*}%
Notice that the case $q:=\infty $ can be easily solved. The proof of
Theorem \ref{14} is now finished .
\end{proof}

\begin{rem}
Let $\alpha :\mathbb{R}^{n}\rightarrow \mathbb{R}$, $p\in \mathcal{P}(%
\mathbb{R}^{n})$ and $q\in \mathcal{P}(\mathbb{R})$, with $1\leq q^{-}\leq
q^{+}<\infty $. Let $\{\mathcal{F}\Phi ,\mathcal{F}\varphi \}$ be a
resolution of unity\textrm{.} We set 
\begin{equation*}
\left\Vert f\right\Vert _{\boldsymbol{B}_{p(\cdot ),q(\cdot )}^{\alpha
(\cdot )}}^{\ast }:=\left\Vert \Phi \ast f\right\Vert _{p(\cdot
)}+\left\Vert \Big(t^{-\frac{1}{q(t)}}\left\Vert t^{-\alpha (\cdot
)}(\varphi _{t}\ast f)\right\Vert _{p(\cdot )}\chi _{\lbrack 2^{-v},2^{1-v}]}%
\Big)_{v}\right\Vert _{\ell _{>}^{q(\cdot )}(L^{q(\cdot )})}.
\end{equation*}%
Then%
\begin{equation}
\left\Vert f\right\Vert _{\boldsymbol{B}_{p(\cdot ),q(\cdot )}^{\alpha
(\cdot )}}\approx \left\Vert f\right\Vert _{\boldsymbol{B}_{p(\cdot
),q(\cdot )}^{\alpha (\cdot )}}^{\ast }.  \label{new-norm1}
\end{equation}
\end{rem}

Now we present the main result in this section.

\begin{thm}
\label{equi-norm}\textit{Let }$\{\mathcal{F}\Phi ,\mathcal{F}\varphi \}$%
\textit{\ be a resolution of unity, }$p\in \mathcal{P}^{\log }\left( 
%TCIMACRO{\U{211d} }%
%BeginExpansion
\mathbb{R}
%EndExpansion
^{n}\right) $, $q\in \mathcal{P}\left( 
%TCIMACRO{\U{211d} }%
%BeginExpansion
\mathbb{R}
%EndExpansion
\right) $ and $\alpha \in C_{\mathrm{loc}}^{\log }\left( 
%TCIMACRO{\U{211d} }%
%BeginExpansion
\mathbb{R}
%EndExpansion
^{n}\right) $.\textit{\ }The Besov space $\boldsymbol{B}_{p(\cdot
),q(0)}^{\alpha (\cdot )}$\ is the collection of all $f\in \mathcal{S}%
^{\prime }(\mathbb{R}^{n})$\ such that 
\begin{equation*}
\left\Vert f\right\Vert _{\boldsymbol{B}_{p(\cdot ),q(0)}^{\alpha (\cdot
)}}:=\left\Vert \Phi \ast f\right\Vert _{p(\cdot )}+\Big(\int_{0}^{1}\left%
\Vert t^{-\alpha (\cdot )}(\varphi _{t}\ast f)\right\Vert _{p(\cdot )}^{q(0)}%
\frac{dt}{t}\Big)^{\frac{1}{q(0)}}<\infty .
\end{equation*}%
\textit{Let }$q$ be \emph{$\log $}-H\"{o}lder continuous at the origin. Then%
\begin{equation*}
\boldsymbol{B}_{p(\cdot ),q(\cdot )}^{\alpha (\cdot )}=\boldsymbol{B}%
_{p(\cdot ),q(0)}^{\alpha (\cdot )},
\end{equation*}%
with equivalent norms.
\end{thm}

\begin{proof} We divide the proof into two steps. \\

\textit{Step 1.} We prove that $\boldsymbol{B}_{p(\cdot
),q(0)}^{\alpha (\cdot )}\hookrightarrow \boldsymbol{B}_{p(\cdot
),q(\cdot )}^{\alpha (\cdot
)}$, which is equivalent to%
\begin{equation*}
\left\Vert f\right\Vert _{\boldsymbol{B}_{p(\cdot ),q(\cdot
)}^{\alpha (\cdot )}}\lesssim \left\Vert f\right\Vert
_{\boldsymbol{B}_{p(\cdot ),q(0)}^{\alpha (\cdot )}}
\end{equation*}
for any $f\in \boldsymbol{B}_{p(\cdot ),q(0)}^{\alpha (\cdot )}$. By
the scaling argument, it suffices to consider the case $\left\Vert
f\right\Vert _{\boldsymbol{B}_{p(\cdot ),q(0)}^{\alpha (\cdot )}}=1$
and show that the modular of $f$ on the left-hand side is bounded.
In
particular, we show that%
\begin{equation*}
S:=\int_{0}^{1}\left\Vert t^{-\alpha (\cdot )}(\varphi _{t}\ast
f)\right\Vert _{p(\cdot )}^{q(t)}\frac{dt}{t}\lesssim 1.
\end{equation*}%
We write%
\begin{equation*}
S=\int_{0}^{\frac{1}{4}}\left\Vert t^{-\alpha (\cdot )}(\varphi
_{t}\ast f)\right\Vert _{p(\cdot )}^{q(t)}\frac{dt}{t}
+\int_{\frac{1}{4}}^{1}\left\Vert t^{-\alpha (\cdot )}(\varphi
_{t}\ast f)\right\Vert _{p(\cdot )}^{q(t)} \frac{dt}{t}=: K+M.
\end{equation*}%
To prove that $M\lesssim 1$, let $\left\{ \mathcal{F}\Psi
,\mathcal{F}\psi \right\} $\ be a resolution of
unity. We find that for any $\frac{1}{4}\leq t\leq 1$\textit{\ }%
\begin{equation*}
\varphi _{t}\ast f=\int_{t/4}^{1}\varphi _{t}\ast \psi _{\tau }\ast f\frac{%
d\tau }{\tau }+\varphi _{t}\ast \Psi \ast f.
\end{equation*}%
Using Lemma \ref{DHR-lemma}, we obtain%
\begin{equation}
t^{-\alpha (x)}\lesssim (1+t^{-1}\left\vert x-y\right\vert
)^{c_{\log }(\alpha )}t^{-\alpha (y)}\lesssim (1+t^{-1}\left\vert
x-y\right\vert )^{c_{\log }(\alpha )}\tau ^{-\alpha (y)}
\label{est-t-and-tau}
\end{equation}%
for any  $\tau \in \lbrack \tfrac{t}{4},\min (1,4t)] $ and $\varphi
\in \mathcal{S}(\mathbb{R}^{n})$, and hence, we get
\begin{equation} \label{EQ:pp}
\left\Vert t^{-\alpha (\cdot )}(\varphi _{t}\ast \psi _{\tau }\ast
f)\right\Vert _{p(\cdot )}\lesssim \left\Vert t^{-\alpha (\cdot
)}(\eta _{t,N}\ast \psi _{\tau }\ast f)\right\Vert _{p(\cdot
)}\lesssim \left\Vert \tau ^{-\alpha (\cdot )}(\psi _{\tau }\ast
f)\right\Vert _{p(\cdot )}
\end{equation}%
for any $   \frac{1}{4}\leq t\leq 1 $, where we used the fact that
the convolution with a radially decreasing $L^{1}$-function is
bounded on $L^{p(\cdot )}$ (by taking $N>0$ large enough).
Similarly, we obtain%
\begin{equation*}
\left\Vert t^{-\alpha (\cdot )}(\varphi _{t}\ast \Psi \ast
f)\right\Vert _{p(\cdot )}\lesssim \left\Vert \Psi \ast f\right\Vert
_{p(\cdot )},\quad \frac{1}{4}\leq t\leq 1.
\end{equation*}%
By using \eqref{EQ:pp} and H\"{o}lder's inequality, we see that
\begin{align*}
\int_{t/4}^{1}\left\Vert t^{-\alpha (\cdot )}(\varphi _{t}\ast \psi
_{\tau }\ast f)\right\Vert _{p(\cdot )}\frac{d\tau }{\tau }
&\lesssim \int_{t/4}^{1}\left\Vert \tau ^{-\alpha (\cdot )}(\psi
_{\tau }\ast
f)\right\Vert _{p(\cdot )}\frac{d\tau }{\tau } \\
&\lesssim \left(\int_{1/16}^{1}\left\Vert \tau ^{-\alpha (\cdot
)}(\psi
_{\tau }\ast f)\right\Vert _{p(\cdot )}^{q(0)}\frac{d\tau }{\tau }\right)^{%
\frac{1}{q(0)}} \\
&\leq 1
\end{align*}%
for any $\frac{1}{4}\leq t\leq 1$. Therefore, we get
\begin{equation*}
M\lesssim \sup_{t\in \lbrack \frac{1}{4},1]}\left\Vert t^{-\alpha
(\cdot )}(\varphi _{t}\ast f)\right\Vert _{p(\cdot )}^{q(t)}\lesssim
1\text{.}
\end{equation*}%

Now we estimate $K$. We write
\begin{equation*}
K=\sum_{v=3}^{\infty }\int_{2^{-v}}^{2^{1-v}}\left\Vert t^{-\alpha
(\cdot )}(\varphi _{t}\ast f)\right\Vert _{p(\cdot
)}^{q(t)}\frac{dt}{t}.
\end{equation*}%
Let us prove that%
\begin{equation}
\int_{2^{-v}}^{2^{1-v}}\left\Vert t^{-\alpha (\cdot )}(\varphi
_{t}\ast f)\right\Vert _{p(\cdot )}^{q(t)}\frac{dt}{t}\lesssim
\int_{2^{-v-2}}^{2^{3-v}}\left\Vert \tau ^{-\alpha (\cdot )}(\psi
_{\tau }\ast f)\right\Vert _{p(\cdot )}^{q(0)}\frac{d\tau }{\tau
}+2^{-v}=\delta \label{fun-est}
\end{equation}%
for any $v\geq 3$. This claim can be reformulated as showing that%
\begin{equation*}
\int_{2^{-v}}^{2^{1-v}}\left(\delta ^{-\frac{1}{q(t)}}\left\Vert
t^{-\alpha
(\cdot )}(\varphi _{t}\ast f)\right\Vert _{p(\cdot )}\right)^{q(t)}\frac{dt}{t}%
\lesssim 1.
\end{equation*}%
We need only to show that%
\begin{equation*}
\delta ^{-\frac{1}{q(t)}}\left\Vert t^{-\alpha (\cdot )}(\varphi
_{t}\ast f)\right\Vert _{p(\cdot )}\lesssim 1.
\end{equation*}%
As before, we find that for any $v\geq 3$ and $t\in \lbrack 2^{-v},2^{1-v}]$%
\begin{equation*}
\varphi _{t}\ast f=\int_{t/4}^{4t}\varphi _{t}\ast \psi _{\tau }\ast f\frac{%
d\tau }{\tau }.
\end{equation*}%
Using $\mathrm{\eqref{est-t-and-tau}}$ and the fact that the
convolution with a
radially decreasing $L^{1}$-function is bounded on $L^{p(\cdot )}$, we obtain%
\begin{equation*}
\left\Vert t^{-\alpha (\cdot )}(\varphi _{t}\ast \psi _{\tau }\ast
f)\right\Vert _{p(\cdot )}\lesssim \left\Vert t^{-\alpha (\cdot
)}(\eta _{t,N}\ast \psi _{\tau }\ast f)\right\Vert _{p(\cdot
)}\lesssim \left\Vert \tau ^{-\alpha (\cdot )}(\psi _{\tau }\ast
f)\right\Vert _{p(\cdot )}\text{.}
\end{equation*}%
Since $q$ is \emph{$\log $-}H\"{o}lder continuous at the origin, we
find
that%
\begin{equation*}
\delta ^{-\frac{1}{q(t)}}\approx \delta ^{-\frac{1}{q(0)}}\text{, }v\geq 3%
\text{ and }t\in \lbrack 2^{-v},2^{1-v}]
\end{equation*}%
with constants independent of $t$ and $v$. Therefore, by taking $N$
large enough and H\"{o}lder's inequality we get
\begin{eqnarray*}
\delta ^{-\frac{1}{q(t)}}\left\Vert t^{-\alpha (\cdot )}(\varphi
_{t}\ast
f)\right\Vert _{p(\cdot )} &\lesssim &\int_{t/4}^{4t}\delta ^{-\frac{1}{q(0)}%
}\left\Vert \tau ^{-\alpha (\cdot )}(\psi _{\tau }\ast f)\right\Vert
_{p(\cdot )}\frac{d\tau }{\tau } \\
&\lesssim &\Big(\int_{t/4}^{4t}\big(\delta
^{-\frac{1}{q(0)}}\left\Vert \tau
^{-\alpha (\cdot )}(\psi _{\tau }\ast f)\right\Vert _{p(\cdot )}\big)^{q(0)}%
\frac{d\tau }{\tau }\Big)^{\frac{1}{q(0)}} \\
&\lesssim &\int_{2^{-v-2}}^{2^{3-v}}\big(\delta
^{-\frac{1}{q(0)}}\left\Vert
\tau ^{-\alpha (\cdot )}(\psi _{\tau }\ast f)\right\Vert _{p(\cdot )}\big)%
^{q(0)}\frac{d\tau }{\tau } \\
&\lesssim &1,
\end{eqnarray*}%
which follows immediately from the definition of $\delta $.
Thus we conclude from \eqref{fun-est} that $K\lesssim 1$.\\

\textit{Step. 2.} We will prove that $\boldsymbol{B}_{p(\cdot
),q(\cdot )}^{\alpha (\cdot )}\hookrightarrow
\boldsymbol{B}_{p(\cdot ),q(0)}^{\alpha
(\cdot )}$, which is equivalent to%
\begin{equation*}
\left\Vert f\right\Vert _{\boldsymbol{B}_{p(\cdot ),q(0)}^{\alpha
(\cdot )}}\lesssim \left\Vert f\right\Vert _{\boldsymbol{B}_{p(\cdot
),q(\cdot )}^{\alpha (\cdot )}}
\end{equation*}
for any $f\in \boldsymbol{B}_{p(\cdot ),q(\cdot )}^{\alpha (\cdot
)}$. By
the scaling argument, we see that it suffices to consider the case $%
\left\Vert f\right\Vert _{\boldsymbol{B}_{p(\cdot ),q(\cdot
)}^{\alpha (\cdot )}}=1$ and show that the modular of $f$ on the
left-hand side is
bounded. In particular, we show that%
\begin{equation*}
H:=\int_{0}^{1}\left\Vert t^{-\alpha (\cdot )}(\varphi _{t}\ast
f)\right\Vert _{p(\cdot )}^{q(0)}\frac{dt}{t}\lesssim 1.
\end{equation*}%
We write%
\begin{equation*}
H=\int_{0}^{\frac{1}{4}}\left\Vert t^{-\alpha (\cdot )}(\varphi
_{t}\ast
f)\right\Vert _{p(\cdot )}^{q(0)} \frac{dt}{t}+\int_{\frac{1}{4}%
}^{1}\left\Vert t^{-\alpha (\cdot )}(\varphi _{t}\ast f)\right\Vert
_{p(\cdot )}^{q(0)} \frac{dt}{t}=:T+D.
\end{equation*}
We use the same notations as in Step 1. We find that
\begin{eqnarray*}
\int_{t/4}^{1}\left\Vert t^{-\alpha (\cdot )}(\varphi _{t}\ast \psi
_{\tau }\ast f)\right\Vert _{p(\cdot )}\frac{d\tau }{\tau }
&\lesssim &\int_{1/16}^{1}\left\Vert \tau ^{-\alpha (\cdot )}(\psi
_{\tau }\ast
f)\right\Vert _{p(\cdot )}\frac{d\tau }{\tau } \\
&\lesssim &\left\Vert f\right\Vert _{\boldsymbol{B}_{p(\cdot
),q(\cdot )}^{\alpha (\cdot )}}\left\Vert \chi _{\lbrack
1/16,1]}\right\Vert
_{q^{\prime }(\cdot )} \\
&\leq &1
\end{eqnarray*}%
for any $\frac{1}{4}\leq t\leq 1$. Therefore, we get $D\lesssim 1$.
To estimate $T$ we need only to prove
\begin{equation*}
\int_{2^{-v}}^{2^{1-v}}\left\Vert t^{-\alpha (\cdot )}(\varphi
_{t}\ast f)\right\Vert _{p(\cdot )}^{q(0)}\frac{dt}{t}\lesssim
\int_{2^{-v-2}}^{2^{3-v}}\left\Vert \tau ^{-\alpha (\cdot )}(\psi
_{\tau }\ast f)\right\Vert _{p(\cdot )}^{q(\tau )}\frac{d\tau }{\tau
}+2^{-v}=\delta
\end{equation*}%
for any $v\geq 3$. This estimate can be obtained by repeating the
above arguments. The proof of Theorem \ref{equi-norm} is complete.
\end{proof}

Lte $a>0$, $\alpha :\mathbb{R}^{n}\rightarrow \mathbb{R}$ and $f\in \mathcal{%
S}^{\prime }(\mathbb{R}^{n})$. Then we define the Peetre maximal function as
follows: 
\begin{equation*}
\varphi _{t}^{\ast ,a}t^{-\alpha (\cdot )}f(x):=\sup_{y\in \mathbb{R}^{n}}%
\frac{t^{-\alpha (y)}\left\vert \varphi _{t}\ast f(y)\right\vert }{\left(
1+t^{-1}\left\vert x-y\right\vert \right) ^{a}},\qquad t>0
\end{equation*}%
and%
\begin{equation*}
\Phi ^{\ast ,a}(x):=\sup_{y\in \mathbb{R}^{n}}\frac{\left\vert \Phi \ast
f(y)\right\vert }{\left( 1+\left\vert x-y\right\vert \right) ^{a}}.
\end{equation*}%
We now present a fundamental characterization of spaces under consideration.

\begin{thm}
\label{fun-char}Let $\alpha \in C_{\mathrm{loc}}^{\log }\left( 
%TCIMACRO{\U{211d} }%
%BeginExpansion
\mathbb{R}
%EndExpansion
^{n}\right) $, $p\in \mathcal{P}^{\log }\left( 
%TCIMACRO{\U{211d} }%
%BeginExpansion
\mathbb{R}
%EndExpansion
^{n}\right) $ and $a>\frac{n}{p^{-}}$. \textit{Let }$q\in \mathcal{P}\left( 
%TCIMACRO{\U{211d} }%
%BeginExpansion
\mathbb{R}
%EndExpansion
\right) $ be \emph{$\log $}-H\"{o}lder continuous at the origin. Let $\{%
\mathcal{F}\Phi ,\mathcal{F}\varphi \}$ be a resolution of unity. Then%
\begin{equation}
\left\Vert f\right\Vert _{\boldsymbol{B}_{p(\cdot ),q(\cdot )}^{\alpha
(\cdot )}}^{\blacktriangledown }:=\left\Vert \Phi ^{\ast ,a}f\right\Vert
_{p(\cdot )}+\left\Vert \left\Vert \varphi _{t}^{\ast ,a}t^{-\alpha (\cdot
)}f\right\Vert _{p(\cdot )}\right\Vert _{L^{q(\cdot )}((0,1],\frac{dt}{t})}
\label{B-equinorm1}
\end{equation}%
\textit{is the equivalent norm in }$\boldsymbol{B}_{p(\cdot ),q(\cdot
)}^{\alpha (\cdot )}$.
\end{thm}

\begin{proof} It is easy to see that for any $f\in \mathcal{S}^{\prime }(%
\mathbb{R}^{n})$ with $\left\Vert f\right\Vert
_{\boldsymbol{B}_{p(\cdot ),q(\cdot )}^{\alpha (\cdot
)}}^{\blacktriangledown }<\infty $ and any $x\in
\mathbb{R}^{n}$ we have%
\begin{equation*}
t^{-\alpha (x)}\left\vert \varphi _{t}\ast f(x)\right\vert \leq
\varphi _{t}^{\ast ,a}t^{-\alpha (\cdot )}f(x)\text{.}
\end{equation*}%
This shows that the right-hand side in $\mathrm{\eqref{B-def}}$ is
less than or equal $\mathrm{\eqref{B-equinorm1}}$.\\

We will prove that there is a constant $C>0$ such that for every
$f\in
\boldsymbol{B}_{p(\cdot ),q(\cdot )}^{\alpha (\cdot )}$%
\begin{equation}
\left\Vert f\right\Vert _{\boldsymbol{B}_{p(\cdot ),q(\cdot
)}^{\alpha
(\cdot )}}^{\blacktriangledown }\leq C\left\Vert f\right\Vert _{\boldsymbol{B%
}_{p(\cdot ),q(\cdot )}^{\alpha (\cdot )}}.  \label{estimate-B}
\end{equation}%
By Lemmas \ref{DHR-lemma} and \ref{r-trick}, the estimate
\begin{eqnarray}
t^{-\alpha (y)}\left\vert \varphi _{t}\ast f(y)\right\vert &\leq
&C_{1}\text{ }t^{-\alpha (y)}\left( \eta _{t,\sigma p^{-}}\ast
|\varphi _{t}\ast
f|^{p^{-}}(y)\right) ^{1/p^{-}}  \notag \\
&\leq &C_{2}\text{ }\left( \eta _{t,(\sigma -c_{\log }(\alpha
))p^{-}}\ast (t^{-\alpha (\cdot )}|\varphi _{t}\ast
f|)^{p^{-}}(y)\right) ^{1/p^{-}} \label{esti-conv}
\end{eqnarray}%
is true for any $y\in \mathbb{R}^{n}$, $\sigma >n/p^{-}$ and $t>0$.
Now dividing both sides of $\mathrm{\eqref{esti-conv}}$ by $\left(
1+t^{-1}\left\vert x-y\right\vert \right) ^{a}$, in the right-hand
side we
use the inequality%
\begin{equation*}
\left( 1+t^{-1}\left\vert x-y\right\vert \right) ^{-a}\leq \left(
1+t^{-1}\left\vert x-z\right\vert \right) ^{-a}\left(
1+t^{-1}\left\vert y-z\right\vert \right) ^{a},\quad x,y,z\in
\mathbb{R}^{n},
\end{equation*}%
while in the left-hand side we take the supremum over $y\in
\mathbb{R}^{n}$, we find that
for all $f\in \boldsymbol{B}_{p(\cdot ),q(\cdot )}^{\alpha (\cdot )}$ any $%
t>0$ and any $\sigma >\max (n/p^{-},a+c_{\log }(\alpha ))$%
\begin{equation*}
\varphi _{t}^{\ast ,a}t^{-\alpha (\cdot )}f(x)\leq C_{2}\text{
}\left( \eta _{t,ap^{-}}\ast (t^{-\alpha (\cdot )}|\varphi _{t}\ast
f|^{p^{-}})(x)\right) ^{1/p^{-}},
\end{equation*}%
where $C_{2}>0$ is independent of $x,t$ and $f$. Applying the $L^{p(\cdot )}$%
-norm and using the fact the convolution with a radially decreasing $L^{1}$%
-function is bounded on $L^{p(\cdot )}$, we get%
\begin{equation*}
\left\Vert \varphi _{t}^{\ast ,a}t^{-\alpha (\cdot
)-1/q(t)}f\right\Vert _{p(\cdot )}\lesssim \left\Vert t^{-\alpha
(\cdot )-1/q(t)}\varphi _{t}\ast f\right\Vert _{p(\cdot )},\quad
t\in (0,1].
\end{equation*}%
Taking $L^{q(\cdot )}((0,1])$-norm, we conclude the desired
estimate. The proof of Theorem \ref{fun-char} is complete.
\end{proof}

In order to formulate the main result of this section, let us consider $%
k_{0},k\in \mathcal{S}(\mathbb{R}^{n})$ and $S\geq -1$ an integer such that
for an $\varepsilon >0$%
\begin{align}
\left\vert \mathcal{F}k_{0}(\xi )\right\vert & >0\text{\quad for\quad }%
\left\vert \xi \right\vert <2\varepsilon,  \label{T-cond1} \\
\left\vert \mathcal{F}k(\xi )\right\vert & >0\text{\quad for\quad }\frac{%
\varepsilon }{2}<\left\vert \xi \right\vert <2\varepsilon  \label{T-cond2}
\end{align}%
and%
\begin{equation}
\int_{\mathbb{R}^{n}}x^{\alpha }k(x)dx=0\text{\quad for any\quad }\left\vert
\alpha \right\vert \leq S.  \label{moment-cond}
\end{equation}%
Here $\mathrm{\eqref{T-cond1}}$ and $\mathrm{\eqref{T-cond2}}$\ are
Tauberian conditions, while $\mathrm{\eqref{moment-cond}}$ states that
moment conditions on $k$. We recall the notation%
\begin{equation*}
k_{t}(x):=t^{-n}k(t^{-1}x)\text{\quad for}\quad t>0.
\end{equation*}%
For any $a>0$, $f\in \mathcal{S}^{\prime }(\mathbb{R}^{n})$ and $x\in 
\mathbb{R}^{n}$ we denote%
\begin{equation}
k_{t}^{\ast ,a}t^{-\alpha \left( \cdot \right) }f(x):=\sup_{y\in \mathbb{R}%
^{n}}\frac{t^{-\alpha \left( y\right) }\left\vert k_{t}\ast f(y)\right\vert 
}{\left( 1+t^{-1}\left\vert x-y\right\vert \right) ^{a}},\quad j\in \mathbb{N%
}_{0}.  \label{Peetre-m-f}
\end{equation}%
Usually $k_{t}\ast f$ is called local mean.\vskip5pt

We are now able to state the so called local mean characterization of $%
\boldsymbol{B}_{p(\cdot ),q(\cdot )}^{\alpha (\cdot )}$ spaces.\vskip5pt

\begin{thm}
\label{loc-mean-char}Let $\alpha \in C_{\mathrm{loc}}^{\log }\left( 
%TCIMACRO{\U{211d} }%
%BeginExpansion
\mathbb{R}
%EndExpansion
^{n}\right) $ and $p\in \mathcal{P}^{\log }\left( 
%TCIMACRO{\U{211d} }%
%BeginExpansion
\mathbb{R}
%EndExpansion
^{n}\right) $. \textit{Let }$q\in \mathcal{P}\left( 
%TCIMACRO{\U{211d} }%
%BeginExpansion
\mathbb{R}
%EndExpansion
\right) $ be \emph{$\log $}-H\"{o}lder continuous at the origin\textit{, }$a>%
\frac{n}{p^{-}}$\textit{\ and }$\alpha ^{+}<S+1$\textit{. Then}%
\begin{equation}
\left\Vert f\right\Vert _{\boldsymbol{B}_{p\left( \cdot \right) ,q\left(
\cdot \right) }^{\alpha \left( \cdot \right) }}^{\prime }:=\left\Vert
k_{0}^{\ast ,a}f\right\Vert _{p(\cdot )}+\left\Vert \left\Vert k_{t}^{\ast
,a}t^{-\alpha (\cdot )}f\right\Vert _{p(\cdot )}\right\Vert _{L^{q(\cdot
)}((0,1],\frac{dt}{t})}  \label{equiv-norms2}
\end{equation}%
\textit{and}%
\begin{equation}
\left\Vert f\right\Vert _{\boldsymbol{B}_{p\left( \cdot \right) ,q\left(
\cdot \right) }^{\alpha \left( \cdot \right) }}^{\prime \prime }:=\left\Vert
k_{0}\ast f\right\Vert _{p(\cdot )}+\left\Vert \left\Vert t^{-\alpha (\cdot
)}(k_{t}\ast f)\right\Vert _{p(\cdot )}\right\Vert _{L^{q(\cdot )}((0,1],%
\frac{dt}{t})},  \label{equiv-norm3}
\end{equation}%
\textit{are equivalent norms on }$\boldsymbol{B}_{p(\cdot ),q(\cdot
)}^{\alpha (\cdot )}$.
\end{thm}

\begin{proof} The idea of the proof is from V. S. Rychkov \cite{Ry01}.
The proof is divided into three steps.\\

\textit{Step }1\textit{.} Let $\varepsilon>0$. Take any pair of functions $\varphi _{0}$ and $%
\varphi \in \mathcal{S}(\mathbb{R}^{n})$ such that%
\begin{align*}
\left\vert \mathcal{F}\varphi _{0}(\xi )\right\vert & >0\quad \text{for}%
\quad \left\vert \xi \right\vert <2\varepsilon, \\
\left\vert \mathcal{F}\varphi (\xi )\right\vert & >0\quad
\text{for}\quad \frac{\varepsilon }{2}<\left\vert \xi \right\vert
<2\varepsilon.
\end{align*}
We prove that there is a constant $c>0$ such
that for any $f\in \boldsymbol{B}_{p(\cdot ),q(\cdot )}^{\alpha (\cdot )}$%
\begin{equation}
\left\Vert f\right\Vert _{\boldsymbol{B}_{p(\cdot ),q(\cdot
)}^{\alpha (\cdot )}}^{\prime }\leq c\left\Vert \varphi _{0}^{\ast
,a}f\right\Vert _{p(\cdot )}+\left\Vert \left\Vert \varphi
_{t}^{\ast ,a}t^{-\alpha (\cdot )}f\right\Vert _{p(\cdot
)}\right\Vert _{L^{q(\cdot )}((0,1],\frac{dt}{t})}. \label{key-est1}
\end{equation}%
Let $\Lambda $, $\lambda \in \mathcal{S}(\mathbb{R}^{n})$ such that%
\begin{equation}
\text{supp }\mathcal{F}\Lambda \subset \{\xi \in
\mathbb{R}^{n}:\left\vert \xi \right\vert <2\varepsilon
\}\text{,\quad supp }\mathcal{F}\lambda \subset \{\xi \in
\mathbb{R}^{n}:\varepsilon /2<\left\vert \xi \right\vert
<2\varepsilon \},  \label{T-cond3}
\end{equation}
\begin{equation*}
\mathcal{F}\Lambda (\xi )\mathcal{F}\varphi _{0}(\xi )+\int_{0}^{1}\mathcal{F%
}\lambda (\tau \xi )\mathcal{F}\varphi (\tau \xi )\frac{d\tau }{\tau }%
=1,\quad \xi \in \mathbb{R}^{n}.
\end{equation*}%
In particular, for any $f\in \boldsymbol{B}_{p(\cdot ),q(\cdot
)}^{\alpha
(\cdot )}$ the following identity is true:%
\begin{equation}
f=\Lambda \ast \varphi _{0}\ast f+\int_{0}^{1}\lambda _{\tau }\ast
\varphi _{\tau }\ast f\frac{d\tau }{\tau }.  \label{Li-P-decom}
\end{equation}%
Hence we can write%
\begin{equation*}
k_{t}\ast f=k_{t}\ast \Lambda \ast \varphi _{0}\ast
f+\int_{0}^{1}k_{t}\ast \lambda _{\tau }\ast \varphi _{\tau }\ast
f\frac{d\tau }{\tau },\quad t\in (0,1].
\end{equation*}%
We have%
\begin{equation}
t^{-\alpha (y)}\left\vert k_{t}\ast \lambda _{\tau }\ast \varphi
_{\tau }\ast f(y)\right\vert \leq t^{-\alpha
(y)}\int_{\mathbb{R}^{n}}\left\vert k_{t}\ast \lambda _{\tau
}(z)\right\vert \left\vert \varphi _{\tau }\ast f(y-z)\right\vert
dz.  \label{est1}
\end{equation}%
First, let $t\leq \tau $. Writing for any $z\in \mathbb{R}^{n}$%
\begin{equation*}
k_{t}\ast \lambda _{\tau }(z)=\tau ^{-n}k_{\frac{t}{\tau }}\ast \lambda (%
\frac{z}{\tau }),
\end{equation*}%
we deduce from Lemma \ref{Ry-Lemma1} that for any integer $S\geq -1$ and any $%
N>0 $ there is a constant $c>0$ independent of $t$ and $\tau $ such
that
\begin{equation*}
\left\vert k_{t}\ast \lambda _{\tau }(z)\right\vert \leq c\text{
}\frac{\tau ^{-n}\left( \frac{t}{\tau }\right) ^{S+1}}{\left( 1+\tau
^{-1}\left\vert z\right\vert \right) ^{N}},\quad z\in
\mathbb{R}^{n}\text{.}
\end{equation*}%
Hence, the right-hand side of $\mathrm{\eqref{est1}}$ is estimated
from above by
\begin{eqnarray*}
&&c\text{ }t^{-\alpha (y)}\tau ^{-n}\left( \frac{t}{\tau }\right)
^{S+1}\int_{\mathbb{R}^{n}}\left( 1+\tau ^{-1}\left\vert
z\right\vert
\right) ^{-N}\left\vert \varphi _{\tau }\ast f(y-z)\right\vert dz \\
&=&c\text{ }\left( \frac{t}{\tau }\right) ^{S+1}t^{-\alpha (y)}\eta
_{\tau ,N}\ast |\varphi _{\tau }\ast f|(y).
\end{eqnarray*}%
By using Lemma \ref{DHR-lemma}, we estimate
\begin{eqnarray*}
t^{-\alpha (y)}\eta _{\tau ,N}\ast |\varphi _{\tau }\ast f|(y) &\leq
&\left( \frac{\tau }{t}\right) ^{\alpha ^{+}}\eta _{\tau ,N-c_{\log
}(\alpha )}\ast
(\tau ^{-\alpha (y)}|\varphi _{v}\ast f|)(y) \\
&\leq &\left( \frac{\tau }{t}\right) ^{\alpha ^{+}}\varphi _{\tau
}^{\ast
,a}\tau ^{-\alpha (\cdot )}f(y)\big\|\eta _{\tau ,N-a-c_{\log }(\alpha )}%
\big\|_{1} \\
&\leq &c\text{ }\left( \frac{\tau }{t}\right) ^{\alpha ^{+}}\varphi
_{\tau }^{\ast ,a}\tau ^{-\alpha (\cdot )}f(y),
\end{eqnarray*}%
for any $N>n+a+c_{\log }(\alpha )$ and any $t\leq \tau $.\\

Next, let $t\geq \tau $. Then, again by Lemma \ref{Ry-Lemma1}, we have for any $%
z\in \mathbb{R}^{n}$ and any $L>0$
\begin{equation*}
\left\vert k_{t}\ast \lambda _{\tau }(z)\right\vert
=t^{-n}\left\vert k\ast
\lambda _{\frac{\tau }{t}}(\frac{z}{t})\right\vert \leq c\text{ }\frac{%
t^{-n}\left( \frac{\tau }{t}\right) ^{M+1}}{\left(
1+t^{-1}\left\vert z\right\vert \right) ^{L}},
\end{equation*}%
where an integer $M\geq -1$ is taken arbitrarily large, since
$D^{\alpha }\mathcal{F}\lambda (0)=0$ for all $\alpha $. Therefore,
for $t\geq \tau $,
the right-hand side of $\mathrm{\eqref{est1}}$ can be estimated from above by%
\begin{eqnarray*}
&&c\text{ }t^{-\alpha (y)}t^{-n}\left( \frac{\tau }{t}\right) ^{M+1}\int_{%
\mathbb{R}^{n}}\left( 1+t^{-1}\left\vert z\right\vert \right)
^{-L}\left\vert \varphi _{\tau }\ast f(y-z)\right\vert dz \\
&=&c\text{ }t^{-\alpha (y)}t^{-n}\left( \frac{\tau }{t}\right)
^{M+1}\eta _{t,L}\ast |\varphi _{\tau }\ast f|(y).
\end{eqnarray*}%
We have for any $t\geq \tau $%
\begin{equation*}
\left( 1+t^{-1}\left\vert z\right\vert \right) ^{-L}\leq \left( \frac{t}{%
\tau }\right) ^{L}\left( 1+\tau ^{-1}\left\vert z\right\vert \right)
^{-L}.
\end{equation*}%
Then, again, the right-hand side of $\mathrm{\eqref{est1}}$ is dominated by%
\begin{eqnarray*}
&&c\text{ }t^{-\alpha (y)}\text{ }\left( \frac{\tau }{t}\right)
^{M-L+1+n}\eta _{\tau ,L}\ast |\varphi _{\tau }\ast f|(y) \\
&\leq &c\text{ }\left( \frac{\tau }{t}\right) ^{M-L+1+n+\alpha
^{-}}\eta _{\tau ,L-c_{\log }(\alpha )}\ast (\tau ^{-\alpha (\cdot
)}|\varphi _{\tau
}\ast f|)(y) \\
&\leq &c\text{ }\left( \frac{\tau }{t}\right) ^{M-L+1+n+\alpha
^{-}}\varphi _{\tau }^{\ast ,a}\tau ^{-\alpha (\cdot
)}f(y)\big\|\eta _{\tau ,L-a-c_{\log
}(\alpha )}\big\|_{1} \\
&\leq &c\text{ }\left( \frac{\tau }{t}\right) ^{M-L+1+n+\alpha
^{-}}\varphi _{\tau }^{\ast ,a}\tau ^{-\alpha (\cdot )}f(y),
\end{eqnarray*}%
where in the first inequality we have used Lemma \ref{DHR-lemma} (by taking $%
L>n+a+c_{\log }(\alpha )$). Let us take $M>L-\alpha ^{-}+a-n$ to
estimate the
last expression by%
\begin{equation*}
c\text{ }\left( \frac{\tau }{t}\right) ^{a+1}\varphi _{\tau }^{\ast
,a}\tau ^{-\alpha (\cdot )}f(y),
\end{equation*}%
where $c>0$ is independent of $t,\tau $ and $f$. Further, note that for all $%
x,y\in \mathbb{R}^{n}$ and all $t,\tau \in (0,1]$%
\begin{eqnarray*}
\varphi _{\tau }^{\ast ,a}\tau ^{-\alpha (\cdot )}f(y) &\leq
&\varphi _{\tau }^{\ast ,a}\tau ^{-\alpha (\cdot )}f(x)(1+\tau
^{-1}\left\vert
x-y\right\vert )^{a} \\
&\leq &\varphi _{\tau }^{\ast ,a}\tau ^{-\alpha (\cdot )}f(x)\max \Big(1,%
\Big(\frac{t}{\tau }\Big)^{a}\Big)(1+t^{-1}\left\vert x-y\right\vert
)^{a}.
\end{eqnarray*}%
Hence%
\begin{equation*}
\sup_{y\in \mathbb{R}^{n}}\frac{t^{-\alpha (y)}\left\vert k_{t}\ast
\lambda _{\tau }\ast \varphi _{\tau }\ast f(y)\right\vert
}{(1+t^{-1}\left\vert x-y\right\vert )^{a}}\leq C\text{ }\varphi
_{\tau }^{\ast ,a}\tau ^{-\alpha (\cdot )}f(x)\times \left\{
\begin{array}{ccc}
\left( \frac{t}{\tau }\right) ^{S+1-\alpha ^{+}} & \text{if} & t\leq \tau, \\
\frac{\tau }{t} & \text{if} & t\geq \tau .%
\end{array}%
\right.
\end{equation*}%
Using the fact that for any $z\in \mathbb{R}^{n}$, any $N>0$ and any
integer
$S\geq -1$%
\begin{equation*}
\left\vert k_{t}\ast \Lambda (z)\right\vert \leq c\text{ }\frac{t^{S+1}}{%
\left( 1+\left\vert z\right\vert \right) ^{N}},
\end{equation*}%
we obtain by the similar arguments that for any $t\in (0,1]$%
\begin{equation*}
\sup_{y\in \mathbb{R}^{n}}\frac{t^{-\alpha (y)}\left\vert k_{t}\ast
\Lambda \ast \varphi _{0}\ast f(y)\right\vert }{(1+t^{-1}\left\vert
x-y\right\vert )^{a}}\leq C\text{ }t^{S+1-\alpha ^{+}}\varphi
_{0}^{\ast ,a}f(x).
\end{equation*}%
Hence for all $f\in \boldsymbol{B}_{p(\cdot ),q(\cdot )}^{\alpha
(\cdot )}$, any $x\in \mathbb{R}^{n}$ and any $t\in (0,1]$, we get
\begin{equation}
k_{t}^{\ast ,a}t^{-\alpha (\cdot )}f(x)\leq C\text{ }t^{S+1-\alpha
^{+}}\varphi _{0}^{\ast ,a}f(x)+C\int_{0}^{1}\min \Big(\Big(\frac{t}{\tau }%
\Big)^{S+1-\alpha ^{+}},\frac{\tau }{t}\Big)\varphi _{\tau }^{\ast
,a}\tau ^{-\alpha (\cdot )}f(x)\frac{d\tau }{\tau }.  \label{est-k}
\end{equation}%
Also we have for any $z\in \mathbb{R}^{n}$, any $N>0$ and any
integer $M\geq
-1$%
\begin{equation*}
\left\vert k_{0}\ast \lambda _{\tau }(z)\right\vert \leq c\text{
}\frac{\tau ^{M+1}}{\left( 1+\left\vert z\right\vert \right) ^{N}}
\end{equation*}%
and%
\begin{equation*}
\left\vert k_{0}\ast \Lambda (z)\right\vert \leq c\text{
}\frac{1}{\left( 1+\left\vert z\right\vert \right) ^{N}}.
\end{equation*}%
As before, we get for any $x\in \mathbb{R}^{n}$%
\begin{equation}
k_{0}^{\ast ,a}f(x)\leq C\text{ }\varphi _{0}^{\ast
,a}f(x)+C\int_{0}^{1}\tau \varphi _{\tau }^{\ast ,a}\tau ^{-\alpha
(\cdot )}f(x)\frac{d\tau }{\tau }.  \label{estk0}
\end{equation}%
In $\mathrm{\eqref{est-k}}$ and $\mathrm{\eqref{estk0}}$ taking the $%
L^{p(\cdot )}$-norm and then using Lemma \ref{lq-inequality}, we get $%
\mathrm{\eqref{key-est1}}$.\\

\textit{Step }2. We prove in this step that there is a constant
$c>0$ such that for any $f\in \boldsymbol{B}_{p(\cdot ),q(\cdot
)}^{\alpha (\cdot )}$
\begin{equation}
\left\Vert f\right\Vert _{\boldsymbol{B}_{p(\cdot ),q(\cdot
)}^{\alpha (\cdot )}}^{\prime }\leq c\left\Vert f\right\Vert
_{\boldsymbol{B}_{p(\cdot ),q(\cdot )}^{\alpha (\cdot )}}^{\prime
\prime }.  \label{key-est2}
\end{equation}%
Analogously to $\mathrm{\eqref{T-cond3}}$,
$\mathrm{\eqref{Li-P-decom}}$ we find two functions $\Lambda $,
$\psi \in \mathcal{S}(\mathbb{R}^{n})$ such that
\begin{equation*}
\mathcal{F}\Lambda (t\xi )\mathcal{F}k_{0}(t\xi )+\int_{0}^{1}\mathcal{F}%
\psi (\tau t\xi )\mathcal{F}k(\tau t\xi )\frac{d\tau }{\tau
}=1,\quad \xi \in \mathbb{R}^{n},
\end{equation*}%
and for all $f\in \boldsymbol{B}_{p(\cdot ),q(\cdot )}^{\alpha
(\cdot )}$ and $t\in (0,1]$
\begin{equation*}
f=\Lambda _{t}\ast \left( k_{0}\right) _{t}\ast f+\int_{0}^{t}\psi
_{h}\ast k_{h}\ast f\frac{dh}{h}.
\end{equation*}%
Hence
\begin{equation*}
k_{t}\ast f=\Lambda _{t}\ast \left( k_{0}\right) _{t}\ast k_{t}\ast
f+\int_{0}^{t}k_{t}\ast \psi _{h}\ast k_{h}\ast f\frac{dh}{h}.
\end{equation*}%
Writing for any $z\in \mathbb{R}^{n}$%
\begin{equation*}
k_{t}\ast \psi _{h}(z)=t^{-n}(k\ast \psi
_{\frac{h}{t}})(\frac{z}{t}),
\end{equation*}%
we deduce from Lemma \ref{Ry-Lemma1} that for any integer $K\geq -1$ and any $%
M>0 $ there is a constant $c>0$ independent of $t$ and $h$ such that
\begin{equation*}
\left\vert k_{t}\ast \psi _{h}(z)\right\vert \leq c\text{ }\frac{%
t^{-n}\left( \frac{h}{t}\right) ^{K+1}}{\left( 1+t^{-1}\left\vert
z\right\vert \right) ^{M}},\quad z\in \mathbb{R}^{n}\text{.}
\end{equation*}%
Analogous estimate%
\begin{equation*}
\left\vert \Lambda _{t}\ast (k_{0})_{t}(z)\right\vert \leq c\text{ }\frac{%
t^{-n}}{\left( 1+t^{-1}\left\vert z\right\vert \right) ^{M}},\quad
z\in \mathbb{R}^{n},
\end{equation*}%
is obvious. From this it follows that
\begin{eqnarray*}
t^{-\alpha \left( y\right) }\left\vert k_{t}\ast f(y)\right\vert &\leq &c%
\text{ }t^{-\alpha \left( y\right) }\eta _{t,M}\ast |k_{t}\ast f|(y) \\
&&+\int_{0}^{t}\left( \frac{h}{t}\right) ^{K+1+\alpha
^{-}}h^{-\alpha \left( y\right) }\eta _{t,M}\ast |k_{h}\ast
f|(y)\frac{dh}{h}.
\end{eqnarray*}%
Since%
\begin{equation*}
\left( 1+t^{-1}\left\vert y-z\right\vert \right) ^{-M}\leq \left( \frac{t}{h}%
\right) ^{M}\left( 1+h^{-1}\left\vert y-z\right\vert \right)
^{-M},\quad y,z\in \mathbb{R}^{n},
\end{equation*}%
by Lemma \ref{DHR-lemma} and H\"{o}lder's inequality, we obtain%
\begin{eqnarray*}
&&t^{-\alpha \left( y\right) }\left\vert k_{t}\ast f(y)\right\vert \\
&\leq &c\text{ }t^{-\alpha \left( y\right) }\eta _{t,M}\ast
|k_{t}\ast f|(y)+\int_{0}^{t}\left( \frac{h}{t}\right) ^{K+1+\alpha
^{-}-M}h^{-\alpha
\left( y\right) }\eta _{h,M}\ast |k_{h}\ast f|(y)\frac{dh}{h} \\
&\leq &c\text{ }\left( \eta _{t,ap^{-}}\ast t^{-\alpha \left( \cdot
\right)
p^{-}}|k_{t}\ast f|^{p^{-}}(y)\right) ^{1/p^{-}} \\
&&+\int_{0}^{t}\left( \frac{h}{t}\right) ^{K+1+\alpha
^{-}-M+n}\left( \eta _{h,ap^{-}}\ast h^{-\alpha \left( \cdot \right)
p^{-}}|k_{h}\ast f|^{p^{-}}(y)\right) ^{1/p^{-}}\frac{dh}{h}
\end{eqnarray*}%
by taking $M>a+n+c_{\log }(\alpha )$.\ Using the elementary inequality:%
\begin{align*}
\left( 1+t^{-1}\left\vert x-y\right\vert \right) ^{-a}& \leq \left(
1+t^{-1}\left\vert x-z\right\vert \right) ^{-a}\left(
1+t^{-1}\left\vert
y-z\right\vert \right) ^{a} \\
& \leq \left( \frac{t}{h}\right) ^{a}\left( 1+h^{-1}\left\vert
x-z\right\vert \right) ^{-a}\left( 1+h^{-1}\left\vert y-z\right\vert
\right) ^{a},
\end{align*}%
and again H\"{o}lder's inequality, we get%
\begin{eqnarray}
\left( k_{t}^{\ast ,a}t^{-\alpha \left( \cdot \right) }f(x)\right)
^{p^{-}} &\leq &c\text{ }\eta _{t,a}\ast t^{-\alpha \left( \cdot
\right)
p^{-}}|k_{t}\ast f|^{p^{-}}(x)  \label{key-est13} \\
&&+c\int_{0}^{t}\left( \frac{h}{t}\right) ^{N}\eta _{h,ap^{-}}\ast
h^{-\alpha \left( \cdot \right) p^{-}}|k_{h}\ast
f|^{p^{-}}(x)\frac{dh}{h}, \notag
\end{eqnarray}%
where $N>0$ can be still be taken arbitrarily large. Similarly, we
obtain
\begin{eqnarray}
\left( k_{0}^{\ast ,a}f(x)\right) ^{p^{-}} &\leq &c\eta
_{1,ap^{-}}\ast
|k_{0}\ast f|(x)+  \label{key-est12} \\
&&\int_{0}^{1}h^{N}\eta _{h,ap^{-}}\ast h^{-\alpha \left( \cdot
\right) p^{-}}|k_{h}\ast f|^{p^{-}}(x)\frac{dh}{h}.  \notag
\end{eqnarray}%
Observe\ that%
\begin{equation*}
\left\Vert \eta _{h,ap^{-}}\ast h^{-\alpha \left( \cdot \right)
p^{-}}|k_{h}\ast f|^{p^{-}}\right\Vert _{p(\cdot )/p^{-}}\lesssim
\left\Vert h^{-\alpha \left( \cdot \right) }k_{h}\ast f\right\Vert
_{p(\cdot )}^{p^{-}},
\end{equation*}%
\begin{equation*}
\left\Vert \eta _{h,ap^{-}}\ast |k_{0}\ast f|^{p^{-}}\right\Vert
_{p(\cdot )/p^{-}}\lesssim \left\Vert k_{0}\ast f\right\Vert
_{p(\cdot )}^{p^{-}},
\end{equation*}%
and
\begin{equation*}
\left\Vert \eta _{t,ap^{-}}\ast t^{-\alpha \left( \cdot \right)
p^{-}}|k_{t}\ast f|^{p^{-}}\right\Vert _{p(\cdot )/p^{-}}\lesssim
\left\Vert t^{-\alpha \left( \cdot \right) }k_{t}\ast f\right\Vert
_{p(\cdot )}^{p^{-}}
\end{equation*}%
for any $h\in (0,t]$ and any $0<t\leq 1$. Then $\mathrm{%
\eqref{key-est12}}$, with power $1/p^{-}$, in $L^{p(\cdot )}$-norm
is bounded by
\begin{equation*}
\left\Vert k_{0}\ast f\right\Vert _{p(\cdot
)}+\int_{0}^{1}h^{N}\left\Vert
h^{-\alpha \left( \cdot \right) }k_{h}\ast f\right\Vert _{p(\cdot )}\frac{dh%
}{h}\lesssim \left\Vert f\right\Vert _{\boldsymbol{B}_{p\left( \cdot
\right) ,q\left( \cdot \right) }^{\alpha \left( \cdot \right)
}}^{\prime \prime },
\end{equation*}%
since $N$ can be still be taken arbitrarily large. In $\mathrm{%
\eqref{key-est13}}$, taking the $L^{p(\cdot )/p^{-}}$-norm, using
the above
estimates, taking the $1/p^{-}$ power and then the $L^{q(\cdot )}((0,1],\frac{%
dt}{t})$-norm, we find from Lemma \ref{lq-inequality} that $$%
\left\Vert f\right\Vert _{\boldsymbol{B}_{p(\cdot ),q(\cdot
)}^{\alpha (\cdot )}}^{\prime }\leq c\left\Vert f\right\Vert
_{\boldsymbol{B}_{p(\cdot ),q(\cdot )}^{\alpha (\cdot )}}^{\prime
\prime }.
$$

\vspace{5mm}

\textit{Step} 3.\ We will prove in this step that for all $f\in \boldsymbol{B%
}_{p(\cdot ),q(\cdot )}^{\alpha (\cdot )}$ the following estimate is true:%
\begin{equation*}
\left\Vert f\right\Vert _{\boldsymbol{B}_{p(\cdot ),q(\cdot
)}^{\alpha (\cdot )}}^{\prime }\leq c\left\Vert f\right\Vert
_{\boldsymbol{B}_{p(\cdot
),q(\cdot )}^{\alpha (\cdot )}}\leq c\left\Vert f\right\Vert _{\boldsymbol{B}%
_{p(\cdot ),q(\cdot )}^{\alpha (\cdot )}}^{\prime \prime }.
\end{equation*}%
Let $\left\{ \mathcal{F}\Phi ,\mathcal{F}\varphi \right\} $ be a
resolution
of unity. The first inequality is proved by the chain of the estimates%
\begin{eqnarray*}
&&\left\Vert f\right\Vert _{\boldsymbol{B}_{p(\cdot ),q(\cdot
)}^{\alpha (\cdot )}}^{\prime }\leq c\left\Vert \Phi ^{\ast
,a}f\right\Vert _{p(\cdot )}+\left\Vert \left\Vert \varphi
_{t}^{\ast ,a}t^{-\alpha (\cdot )}f\right\Vert _{p(\cdot
)}\right\Vert _{L^{q(\cdot )}((0,1],\frac{dt}{t})}
\\
&\leq &c\left\Vert \Phi \ast f\right\Vert _{p(\cdot )}+\left\Vert
\left\Vert t^{-\alpha (\cdot )}\varphi _{t}\ast f\right\Vert
_{p(\cdot )}\right\Vert
_{L^{q(\cdot )}((0,1],\frac{dt}{t})}\leq c\left\Vert f\right\Vert _{%
\boldsymbol{B}_{p(\cdot ),q(\cdot )}^{\alpha (\cdot )}},
\end{eqnarray*}%
where the first inequality is $\mathrm{\eqref{key-est1}}$, see Step
1, the
second inequality is $\mathrm{\eqref{key-est2}}$ (with $\varphi $ and $%
\varphi _{0}$ instead of $k$ and $k_{0}$), see Step 2, and finally
the third inequality is obvious. Now the second inequality can be
obtained by the
following chain%
\begin{eqnarray*}
\left\Vert f\right\Vert _{\boldsymbol{B}_{p(\cdot ),q(\cdot
)}^{\alpha (\cdot )}} &\leq &c\left\Vert \Phi ^{\ast ,a}f\right\Vert
_{p(\cdot )}+\left\Vert \left\Vert \varphi _{t}^{\ast ,a}t^{-\alpha
(\cdot
)}f\right\Vert _{p(\cdot )}\right\Vert _{L^{q(\cdot )}(\left[ 0,1\right] ,%
\frac{dt}{t})} \\
&\leq &c\left\Vert f\right\Vert _{\boldsymbol{B}_{p(\cdot ),q(\cdot
)}^{\alpha (\cdot )}}^{\prime }\leq c\left\Vert f\right\Vert _{\boldsymbol{B}%
_{p(\cdot ),q(\cdot )}^{\alpha (\cdot )}}^{\prime \prime },
\end{eqnarray*}%
where the first inequality is obvious, the second inequality is $\mathrm{%
\eqref{key-est1}}$ in Step 1, with the roles of $k_{0}$ and $k$
respectively $\varphi _{0}$ and $\varphi $ interchanged, and finally
the last inequality is $\mathrm{\eqref{key-est2}}$ in Step 2. Thus,
Theorem \ref{loc-mean-char} is proved.
\end{proof}

From Theorems \ref{equi-norm} and \ref{loc-mean-char}, we have:

\begin{thm}
\label{loc-mean-char1}Let $\alpha \in C_{\mathrm{loc}}^{\log }\left( 
%TCIMACRO{\U{211d} }%
%BeginExpansion
\mathbb{R}
%EndExpansion
^{n}\right) $ and $p\in \mathcal{P}^{\log }\left( 
%TCIMACRO{\U{211d} }%
%BeginExpansion
\mathbb{R}
%EndExpansion
^{n}\right) $. Let $q\in \mathcal{P}\left( 
%TCIMACRO{\U{211d} }%
%BeginExpansion
\mathbb{R}
%EndExpansion
\right) $ be \emph{$\log $}-H\"{o}lder continuous at the origin\textit{, }$a>%
\frac{n}{p^{-}}$\textit{\ and }$\alpha ^{+}<S+1$\textit{. Let }$f\in 
\boldsymbol{B}_{p(\cdot ),q(\cdot )}^{\alpha (\cdot )}$,\textit{\ }%
\begin{equation*}
\left\Vert f\right\Vert _{\boldsymbol{B}_{p\left( \cdot \right) ,q\left(
0\right) }^{\alpha \left( \cdot \right) }}^{\prime }:=\left\Vert k_{0}^{\ast
,a}f\right\Vert _{p(\cdot )}+\left\Vert \left\Vert k_{t}^{\ast ,a}t^{-\alpha
(\cdot )}f\right\Vert _{p(\cdot )}\right\Vert _{L^{q(0)}((0,1],\frac{dt}{t})}
\end{equation*}%
and%
\begin{equation*}
\left\Vert f\right\Vert _{\boldsymbol{B}_{p\left( \cdot \right) ,q\left(
0\right) }^{\alpha \left( \cdot \right) }}^{\prime \prime }:=\left\Vert
k_{0}\ast f\right\Vert _{p(\cdot )}+\left\Vert \left\Vert t^{-\alpha (\cdot
)}(k_{t}\ast f)\right\Vert _{p(\cdot )}\right\Vert _{L^{q(0)}((0,1],\frac{dt%
}{t})}.
\end{equation*}%
\textit{Then}%
\begin{equation*}
\left\Vert f\right\Vert _{\boldsymbol{B}_{p\left( \cdot \right) ,q\left(
\cdot \right) }^{\alpha \left( \cdot \right) }}^{\prime }\approx \left\Vert
f\right\Vert _{\boldsymbol{B}_{p\left( \cdot \right) ,q\left( 0\right)
}^{\alpha \left( \cdot \right) }}^{\prime }\approx \left\Vert f\right\Vert _{%
\boldsymbol{B}_{p\left( \cdot \right) ,q\left( \cdot \right) }^{\alpha
\left( \cdot \right) }}^{\prime \prime }\approx \left\Vert f\right\Vert _{%
\boldsymbol{B}_{p\left( \cdot \right) ,q\left( 0\right) }^{\alpha \left(
\cdot \right) }}^{\prime \prime }\approx \left\Vert f\right\Vert _{%
\boldsymbol{B}_{p\left( \cdot \right) ,q\left( .\right) }^{\alpha \left(
\cdot \right) }}.
\end{equation*}
\end{thm}

\section{Embeddings}

For the spaces $\boldsymbol{B}_{p(\cdot ),q(\cdot )}^{\alpha (\cdot )}$
introduced above we want to show some embedding theorems. We say a
quasi-Banach space $A_{1}$ is continuously embedded in another quasi-Banach
space $A_{2}$, if $A_{1}\subset A_{2}$ and there exists a constant $c>0$
such that 
\begin{equation*}
\left\Vert f\right\Vert _{A_{2}}\leq c\left\Vert f\right\Vert _{A_{1}}
\end{equation*}
for all $f\in A_{1}$. Then we write 
\begin{equation*}
A_{1}\hookrightarrow A_{2}.
\end{equation*}
We begin with the following elementary embeddings.

\begin{thm}
\label{Elem-emb} Let $\alpha _{0},\alpha _{1}\in C_{\mathrm{loc}}^{\log
}\left( 
%TCIMACRO{\U{211d} }%
%BeginExpansion
\mathbb{R}
%EndExpansion
^{n}\right) $ and $p\in \mathcal{P}^{\log }\left( 
%TCIMACRO{\U{211d} }%
%BeginExpansion
\mathbb{R}
%EndExpansion
^{n}\right) $. Let $q_{0},q_{1}\in \mathcal{P}\left( 
%TCIMACRO{\U{211d} }%
%BeginExpansion
\mathbb{R}
%EndExpansion
\right) $ satisfy the log-H\"{o}lder decay condition at the origin. If\ $%
(\alpha _{0}-\alpha _{1})^{-}>0$, then 
\begin{equation*}
\boldsymbol{B}_{{p(\cdot )},q_{0}{(\cdot )}}^{\alpha _{0}(\cdot
)}\hookrightarrow \boldsymbol{B}_{{p(\cdot )},q_{1}{(\cdot )}}^{\alpha
_{1}(\cdot )}.
\end{equation*}
\end{thm}

\begin{proof} Let $\{\mathcal{F}\Phi, \mathcal{F}\varphi \}$ and
$\left\{
\mathcal{F}\Psi ,\mathcal{F}\psi \right\} $ be two resolutions of unity.%
\textit{\ }We have%
\begin{eqnarray*}
\varphi _{t}\ast f &=&\int_{t/4}^{\min (1,4t)}\varphi _{t}\ast \psi
_{\tau
}\ast f\frac{d\tau }{\tau }+h_{t},\quad t\in (0,1] \\
&=&g_{t}+h_{t},
\end{eqnarray*}%
where%
\begin{equation*}
h_{t}=\left\{
\begin{array}{ccc}
0, & \text{if} & 0<t<\frac{1}{4}; \\
\varphi _{t}\ast \Psi \ast f, & \text{if} & \frac{1}{4}\leq t\leq 1.%
\end{array}%
\right.
\end{equation*}%
From the fact that the convolution with a radially decreasing
$L^{1}$-function is bounded on $L^{p(\cdot )}$, we deduce that
\begin{eqnarray*}
&&\left\Vert t^{-{\alpha }_{1}{(\cdot )}}g_{t}\right\Vert _{{p(\cdot )}} \\
&\lesssim &t^{(\alpha _{0}-\alpha _{1})^{-}}\int_{t/4}^{\min
(1,4t)}\left\Vert \tau ^{-{\alpha }_{0}{(\cdot )}}(\psi _{\tau }\ast
f)\right\Vert _{{p(\cdot )}}\frac{d\tau }{\tau } \\
&\lesssim &t^{(\alpha _{0}-\alpha _{1})^{-}}\left\Vert \left\Vert
\tau ^{-\alpha _{0}(\cdot )}(\psi _{\tau }\ast f)\right\Vert
_{p(\cdot )}\right\Vert _{L^{q_{0}(\cdot )}((0,1],\frac{d\tau }{\tau
})}\left\Vert
1\right\Vert _{L^{q_{0}^{\prime }(\cdot )}([t/4,\min (1,4t)],\frac{d\tau }{%
\tau })} \\
&\lesssim &t^{(\alpha _{0}-\alpha _{1})^{-}}\left\Vert \left\Vert
\tau ^{-\alpha _{0}(\cdot )}(\psi _{\tau }\ast f)\right\Vert
_{p(\cdot )}\right\Vert _{L^{q_{0}(\cdot )}((0,1],\frac{d\tau }{\tau
})}\lesssim
t^{(\alpha _{0}-\alpha _{1})^{-}}\left\Vert f\right\Vert _{\boldsymbol{B}_{{%
p(\cdot )},q_{0}{(\cdot )}}^{{\alpha }_{0}{(\cdot )}}},
\end{eqnarray*}%
where we have used H\"{o}lder's inequality. Similarly, we find that%
\begin{equation*}
\left\Vert t^{-{\alpha }_{1}{(\cdot )}}h_{t}\right\Vert _{{p(\cdot )}%
}\lesssim \left\{
\begin{array}{ccc}
0, & \text{if} & 0<t<\frac{1}{4}; \\
\left\Vert \Psi \ast f\right\Vert _{{p(\cdot )}}, & \text{if} & \frac{1}{4}%
\leq t\leq 1.%
\end{array}%
\right.
\end{equation*}%
Taking the $L^{q_{1}(\cdot )}((0,1],\frac{dt}{t})$-norm we obtain
the
desired estimate. The estimation of $\left\Vert \Phi \ast f\right\Vert _{{%
p(\cdot )}}$ can be obtained by the decomposition
$\mathrm{\eqref{dec-phi}}$. The proof of Theorem \ref{Elem-emb} is
complete.
\end{proof}

\begin{thm}
\label{Elem-emb1}Let $\alpha \in C_{\mathrm{loc}}^{\log }\left( 
%TCIMACRO{\U{211d} }%
%BeginExpansion
\mathbb{R}
%EndExpansion
^{n}\right) $ and $p\in \mathcal{P}^{\log }\left( 
%TCIMACRO{\U{211d} }%
%BeginExpansion
\mathbb{R}
%EndExpansion
^{n}\right) $. Let $q_{0},q_{1}\in \mathcal{P}\left( 
%TCIMACRO{\U{211d} }%
%BeginExpansion
\mathbb{R}
%EndExpansion
\right) $ satisfy the $\log$-H\"{o}lder decay condition at the origin. If\ $%
q_{0}{(0)}\leq q_{1}{(0)}$, then 
\begin{equation*}
\boldsymbol{B}_{{p(\cdot )},q_{0}{(\cdot )}}^{\alpha {(\cdot )}%
}\hookrightarrow \boldsymbol{B}_{{p(\cdot )},q_{1}{(\cdot )}}^{\alpha {%
(\cdot )}}.
\end{equation*}
\end{thm}

\begin{proof} Let $f\in \boldsymbol{B}_{{p(\cdot )},q_{0}{(\cdot )}%
}^{\alpha {(\cdot )}}$. By the scaling argument, it suffices to
consider the case $\left\Vert f\right\Vert _{\boldsymbol{B}_{{p(\cdot )}%
,q_{0}{(\cdot )}}^{\alpha {(\cdot )}}}=1$ and show that the modular
of $f$ on the left-hand side is bounded. In particular, by Theorem
\ref{equi-norm},
we show that%
\begin{equation*}
\int_{0}^{1}\left\Vert t^{-\alpha (\cdot )}(\varphi _{t}\ast
f)\right\Vert _{p(\cdot )}^{q_{1}{(0)}}\frac{dt}{t}\lesssim 1.
\end{equation*}%
We use the notation in the previous theorem. We need only to prove that%
\begin{equation*}
\left\Vert t^{-\alpha (\cdot )}(\varphi _{t}\ast f)\right\Vert
_{p(\cdot )}^{q_{1}{(0)-}q_{0}{(0)}}\lesssim 1
\end{equation*}%
for any $t\in (0,1]$. We have%
\begin{equation*}
\varphi _{t}\ast f=\int_{t/4}^{\min (1,4t)}\varphi _{t}\ast \psi
_{\tau }\ast f\frac{d\tau }{\tau }+h_{t},\quad t\in (0,1].
\end{equation*}%
From Lemma \ref{DHR-lemma} and the fact the convolution with a
radially decreasing $L^{1}$-function is bounded on $L^{p(\cdot )}$,
we deduce that
\begin{equation*}
\left\Vert t^{-\alpha (\cdot )}(\varphi _{t}\ast f)\right\Vert
_{p(\cdot
)}\lesssim \int_{t/4}^{\min (1,4t)}\left\Vert \tau ^{-{\alpha (\cdot )}%
}(\psi _{\tau }\ast f)\right\Vert _{{p(\cdot )}}\frac{d\tau }{\tau }%
+\left\Vert t^{-{\alpha (\cdot )}}h_{t}\right\Vert _{{p(\cdot )}},
\end{equation*}%
where%
\begin{equation*}
\left\Vert t^{-{\alpha (\cdot )}}h_{t}\right\Vert _{{p(\cdot
)}}\lesssim \left\{
\begin{array}{ccc}
0, & \text{if} & 0<t<\frac{1}{4}; \\
\left\Vert \Psi \ast f\right\Vert _{{p(\cdot )}}, & \text{if} & \frac{1}{4}%
\leq t\leq 1,%
\end{array}%
\right.
\end{equation*}%
which implies that
$$\left\Vert t^{-{\alpha (\cdot )}}h_{t}\right\Vert _{{p(\cdot )}%
}\lesssim 1$$ for any $t\in (0,1]$. Finally, by using H\"{o}lder's
inequality we find that
\begin{equation*}
\int_{t/4}^{\min (1,4t)}\left\Vert \tau ^{-{\alpha (\cdot )}}(\psi
_{\tau
}\ast f)\right\Vert _{{p(\cdot )}}\frac{d\tau }{\tau }\lesssim \Big(%
\int_{t/4}^{\min (1,4t)}\left\Vert \tau ^{-{\alpha (\cdot )}}(\psi
_{\tau
}\ast f)\right\Vert _{{p(\cdot )}}^{q_{0}{(0)}}\frac{d\tau }{\tau }\Big)^{%
\frac{1}{q_{0}{(0)}}}\lesssim 1,
\end{equation*}%
and hence, Theorem \ref{Elem-emb1} is proved.
\end{proof}

We next consider embeddings of Sobolev type. It is well known that%
\begin{equation*}
B_{{p}_{0},q}^{{\alpha }_{0}}\hookrightarrow B_{{p}_{1},q}^{{\alpha }_{1}},
\end{equation*}%
if ${\alpha }_{0}-n/{p}_{0}={\alpha }_{1}-n/{p}_{1}$, where $0<{p}_{0}<{p}%
_{1}\leq \infty $ and $0<q\leq \infty $ (see e.g. \cite[Theorem 2.7.1]{T1}).
In the following theorem we generalize these embeddings to variable exponent
case.

\begin{thm}
\label{Sobolev-emb} Let $\alpha _{0},\alpha _{1}\in C_{\mathrm{loc}}^{\log
}\left( 
%TCIMACRO{\U{211d} }%
%BeginExpansion
\mathbb{R}
%EndExpansion
^{n}\right) $ and $p_{0},p_{1}\in \mathcal{P}^{\log }\left( 
%TCIMACRO{\U{211d} }%
%BeginExpansion
\mathbb{R}
%EndExpansion
^{n}\right) $. Let $q\in \mathcal{P}\left( 
%TCIMACRO{\U{211d} }%
%BeginExpansion
\mathbb{R}
%EndExpansion
\right) $ be \emph{$\log $}-H\"{o}lder continuous at the origin. If ${\alpha 
}_{0}\geq {\alpha }_{1}$ and ${\alpha }_{0}{(x)-}\frac{n}{p_{0}(x)}={\alpha }%
_{1}{(x)-}\frac{n}{p_{1}(x)}$, then 
\begin{equation}
\boldsymbol{B}_{{p}_{0}{(\cdot )},q{(\cdot )}}^{{\alpha }_{0}{(\cdot )}%
}\hookrightarrow \boldsymbol{B}_{{p}_{1}{(\cdot )},q{(\cdot )}}^{{\alpha }%
_{1}{(\cdot )}}.  \label{Sobolev-emd}
\end{equation}
\end{thm}

\begin{proof} We prove that%
\begin{equation*}
\left\Vert f\right\Vert _{\boldsymbol{B}_{{p}_{1}{(\cdot )},q{(\cdot )}}^{{%
\alpha }_{1}{(\cdot )}}}\lesssim \left\Vert f\right\Vert _{\boldsymbol{B}_{{p%
}_{0}{(\cdot )},q{(\cdot )}}^{{\alpha }_{0}{(\cdot )}}}
\end{equation*}%
for any $f\in \boldsymbol{B}_{{p}_{0}{(\cdot )},q{(\cdot )}}^{{\alpha }_{0}{%
(\cdot )}}$. Let us prove that%
\begin{equation}
\left\Vert t^{-{\alpha }_{1}{(\cdot )}}(\varphi _{t}\ast f)\right\Vert _{{p}%
_{1}{(\cdot )}}\lesssim \left\Vert t^{-{\alpha }_{0}{(\cdot
)}}(\varphi _{t}\ast f)\right\Vert _{{p}_{0}{(\cdot )}}=\delta .
\label{key-est}
\end{equation}%
This is equivalent to%
\begin{equation*}
\left\Vert \delta ^{-1}t^{-{\alpha }_{1}{(\cdot )}}(\varphi _{t}\ast
f)\right\Vert _{{p}_{1}{(\cdot )}}\lesssim 1.
\end{equation*}%
By Lemma \ref{r-trick} we have for any $m>n$, $d>0$
\begin{equation*}
|\varphi _{t}\ast f(x)|\leq c(\eta _{t,m}\ast |\varphi _{t}\ast
f|^{d}(x))^{1/d}.
\end{equation*}%
Hence%
\begin{eqnarray*}
&&\delta ^{-1}t^{{\frac{n}{{p}_{1}{(x)}}}-{\alpha
}_{1}{(x)}}|\varphi
_{t}\ast f(x)| \\
&\leq &c(\eta _{t,m-c_{\log }(\alpha _{1})-c_{\log }(1/p_{1})}\ast
\delta ^{-d}t^{-{\alpha }_{1}{(\cdot )d+\frac{nd}{{p}_{1}{(\cdot
)}}}}|\varphi
_{t}\ast f|^{d}(x))^{1/d} \\
&\leq &c\left\Vert t^{n-\frac{n}{d}+\frac{n}{{p}_{0}{(\cdot )}}}\eta
_{t,(m-c_{\log }(\alpha _{1})-c_{\log }(1/p_{1}))/d}(x-\cdot )\right\Vert _{h%
{(\cdot )}}\left\Vert \delta ^{-d}t^{-{\alpha }_{0}{(\cdot
)}}(\varphi _{t}\ast f)\right\Vert _{{p}_{0}{(\cdot )}}.
\end{eqnarray*}%
where $$\frac{1}{d}=\frac{1}{p_{0}{(\cdot )}}+\frac{1}{h{(\cdot
)}}.$$ The
second norm on the right hand side is bounded by $1$ due to the choice of $%
\delta $. To show that the first norm is also bounded, we
investigate the
corresponding modular:%
\begin{eqnarray*}
&&\varrho _{h(\cdot )}(t^{n-\frac{n}{d}+\frac{n}{{p}_{0}{(\cdot
)}}}\eta
_{t,(m-c_{\log }(\alpha _{1})-c_{\log }(1/p_{1}))/d}(x-\cdot )) \\
&=&\int_{\mathbb{R}^{n}}\frac{t^{-n}}{\left( 1+t^{-1}\left\vert
x-z\right\vert \right) ^{(m-c_{\log }(\alpha _{1})-c_{\log }(1/p_{1}))h{(z)}%
/d}}dz<\infty
\end{eqnarray*}%
for $m>0$ large enough. Now we have
\begin{eqnarray*}
&&\left\vert \delta ^{-1}t^{-{\alpha }_{1}{(x)}}\varphi _{t}\ast
f(x)\right\vert ^{{p}_{1}{(x)}} \\
&=&\left\vert \delta ^{-1}t^{\frac{n}{{p}_{1}{(x)}}-{\alpha }_{1}{(x)}%
}\varphi _{t}\ast f(x)\right\vert
^{{p}_{1}{(x)-p}_{0}{(x)}}\left\vert
\delta ^{-1}t^{-{\alpha }_{0}{(x)}}\varphi _{t}\ast f(x)\right\vert ^{{p}_{0}%
{(x)}} \\
&\lesssim &\left\vert \delta ^{-1}t^{-{\alpha }_{0}{(x)}}\varphi
_{t}\ast f(x)\right\vert ^{{p}_{0}{(x)}}.
\end{eqnarray*}%
Integrating this inequality over $%
%TCIMACRO{\U{211d} }%
%BeginExpansion
\mathbb{R}^{n}$ and taking into account the definition of $\delta $,
we conclude that the claim is true. The proof of Theorem
\ref{Sobolev-emb} is complete by taking in
$\mathrm{\eqref{key-est}}$
the $%
L^{q(\cdot )}((0,1],\frac{dt}{t})$-norm.
\end{proof}

Let $\alpha \in C_{\mathrm{loc}}^{\log }\left( 
%TCIMACRO{\U{211d} }%
%BeginExpansion
\mathbb{R}
%EndExpansion
^{n}\right) $ and $p\in \mathcal{P}^{\log }\left( 
%TCIMACRO{\U{211d} }%
%BeginExpansion
\mathbb{R}
%EndExpansion
^{n}\right) $. Let $q\in \mathcal{P}\left( 
%TCIMACRO{\U{211d} }%
%BeginExpansion
\mathbb{R}
%EndExpansion
\right) $ be \emph{$\log $}-H\"{o}lder continuous at the origin. From $%
\mathrm{\eqref{Sobolev-emd}}$, we obtain%
\begin{equation*}
\boldsymbol{B}_{p(\cdot ),q(\cdot )}^{\alpha (\cdot )}\hookrightarrow 
\boldsymbol{B}_{p^{+},q(\cdot )}^{\alpha (\cdot )+n/p^{+}-n/p(\cdot
)}\hookrightarrow \boldsymbol{B}_{p^{+},q(\cdot )}^{(\alpha
+n/p^{+}-n/p)^{-}}\hookrightarrow \boldsymbol{B}_{p^{+},\infty }^{(\alpha
+n/p^{+}-n/p)^{-}-\varepsilon }\hookrightarrow \mathcal{S}^{\prime }(\mathbb{%
R}^{n}),
\end{equation*}%
where $0<\varepsilon <(\alpha +n/p^{+}-n/p)^{-}$. Let $\alpha _{0}\in 
%TCIMACRO{\U{211d} }%
%BeginExpansion
\mathbb{R}
%EndExpansion
$ be such that $\alpha _{0}>(\alpha +n/p^{-}-n/p)^{+}$. We have%
\begin{equation*}
\mathcal{S}(\mathbb{R}^{n})\hookrightarrow \boldsymbol{B}_{p^{-},q^{+}}^{%
\alpha _{0}}\hookrightarrow \boldsymbol{B}_{p^{-},q(\cdot )}^{\alpha
_{0}-n/p^{-}+n/p(\cdot )}\hookrightarrow \boldsymbol{B}_{p(\cdot ),q(\cdot
)}^{\alpha (\cdot )}.
\end{equation*}%
Thus we have:

\begin{thm}
Let $\alpha \in C_{\mathrm{loc}}^{\log }\left( 
%TCIMACRO{\U{211d} }%
%BeginExpansion
\mathbb{R}
%EndExpansion
^{n}\right) $ and $p\in \mathcal{P}^{\log }\left( 
%TCIMACRO{\U{211d} }%
%BeginExpansion
\mathbb{R}
%EndExpansion
^{n}\right) $. Let $q\in \mathcal{P}\left( 
%TCIMACRO{\U{211d} }%
%BeginExpansion
\mathbb{R}
%EndExpansion
\right) $ be \emph{$\log $}-H\"{o}lder continuous at the origin. Then 
\begin{equation*}
\mathcal{S}(\mathbb{R}^{n})\hookrightarrow \boldsymbol{B}_{p(\cdot ),q(\cdot
)}^{\alpha (\cdot )}\hookrightarrow \mathcal{S}^{\prime }(\mathbb{R}^{n}).
\end{equation*}
\end{thm}

\section{Atomic decomposition}

The idea of atomic decompositions goes back to M. Frazier and B. Jawerth in
a series of papers \cite{FJ86}, \cite{FJ90} (see also \cite{T3}). The main
goal of this section is to prove an atomic decomposition result for $%
\boldsymbol{B}_{p(\cdot ),q(\cdot )}^{\alpha (\cdot )}$. Atoms are the
building blocks for the atomic decomposition.

\begin{defn}
\label{Atom-Def} Let $K, \, L+1 \in \mathbb{N}_{0}$ and $\gamma >1$. A $K$%
-times continuous differentiable function $a\in C^{K}(\mathbb{R}^{n})$ is
called $[K,L]$-atom centered at $Q_{v,m}$, $v\in \mathbb{N}_{0}$ and $m\in 
\mathbb{Z}^{n}$, if 
\begin{equation}
\mathrm{supp}\text{ }a\subseteq \gamma Q_{v,m},  \label{supp-cond}
\end{equation}
\begin{equation}
|D^{\beta }a(x)|\leq 2^{v(|\beta |+1/2)}\text{,\quad for\quad }0\leq |\beta
|\leq K,x\in \mathbb{R}^{n} ,  \label{diff-cond}
\end{equation}%
and 
\begin{equation}
\int_{\mathbb{R}^{n}}x^{\beta }a(x)dx=0 \text{\quad for\quad }0\leq |\beta
|\leq L\text{ and }v\geq 1.  \label{mom-cond}
\end{equation}
\end{defn}

If the atom $a$ is located at $Q_{v,m}$, i.e., if it fulfills $\mathrm{%
\eqref{supp-cond}}$, then we denote it by $a_{v,m}$. For $v=0$ or $L=-1$ the
moment\ conditions\ $\mathrm{\eqref{mom-cond}}$ is not required.\newline

For proving the decomposition by atoms we need the following lemma, see
Frazier and Jawerth \cite[Lemma 3.3]{FJ86}.

\begin{lem}
\label{FJ-lemma}Let $\{\mathcal{F}\Phi ,\mathcal{F}\varphi \}$ be a
resolution of unity and let $\varrho _{v,m}$ be an $\left[ K,L\right] $%
-atom. If $j\in \mathbb{N}_{0}$ and $2^{-j}\leq t\leq 2^{1-j}$, then%
\begin{equation*}
\left\vert \varphi _{t}\ast \varrho _{v,m}(x)\right\vert \leq c\text{ }%
2^{(v-j)K+vn/2}\left( 1+2^{v}\left\vert x-x_{Q_{v,m}}\right\vert \right)
^{-M}
\end{equation*}%
if $v\leq j$ and%
\begin{equation*}
\left\vert \varphi _{t}\ast \varrho _{v,m}(x)\right\vert \leq c\text{ }%
2^{(j-v)(L+n+1)+vn/2}\left( 1+2^{j}\left\vert x-x_{Q_{v,m}}\right\vert
\right) ^{-M}
\end{equation*}%
if $v\geq j$, where $M$ is sufficiently large and $\varphi
_{t}=t^{-n}\varphi (\frac{\cdot }{t})$. Moreover%
\begin{equation*}
\left\vert \Phi \ast \varrho _{v,m}(x)\right\vert \leq c\text{ }%
2^{-v(L+n+1)+vn/2}\left( 1+\left\vert x-x_{Q_{v,m}}\right\vert \right) ^{-M}.
\end{equation*}
\end{lem}

Let $p\in \mathcal{P}\left( 
%TCIMACRO{\U{211d} }%
%BeginExpansion
\mathbb{R}
%EndExpansion
^{n}\right) $, $q\in \mathcal{P}\left( 
%TCIMACRO{\U{211d} }%
%BeginExpansion
\mathbb{R}
%EndExpansion
\right) $\ and $\alpha $ $:\mathbb{R}^{n}\rightarrow \mathbb{R}$. Then for
any complex valued sequences $\lambda =\{\lambda _{v,m}\in \mathbb{C}%
\}_{v\in \mathbb{N}_{0},m\in \mathbb{Z}^{n}}$ we define%
\begin{equation*}
\boldsymbol{b}_{p(\cdot ),q(\cdot )}^{\alpha (\cdot )}:=\Big\{\lambda
:\left\Vert \lambda \right\Vert _{\boldsymbol{b}_{p(\cdot ),q(\cdot
)}^{\alpha (\cdot )}}<\infty \Big\},
\end{equation*}%
where%
\begin{eqnarray*}
\left\Vert \lambda \right\Vert _{\boldsymbol{b}_{p(\cdot ),q(\cdot
)}^{\alpha (\cdot )}} \text{:=}&& \big\|\sum\limits_{m\in \mathbb{Z}%
^{n}}\lambda _{0,m}\chi _{0,m}\big\|_{p(\cdot )} \\
&& +\Big\|\Big(\big\|t^{-(\alpha \left( \cdot \right) +n/2)-\frac{1}{q(t)}%
}\sum\limits_{m\in \mathbb{Z}^{n}}\lambda _{v,m}\chi _{v,m}\big\|_{p(\cdot
)}\chi _{\left[ 2^{-v},2^{1-v}\right] }\Big)_{v}\Big\|_{\ell _{>}^{q(\cdot
)}(L^{q(\cdot )})}.
\end{eqnarray*}%
Here $\chi _{v,m}$ is the characteristic function of the cube $Q_{v,m}$. Let 
$p\in \mathcal{P}(\mathbb{R}^{n})$ and $\alpha \in C_{\mathrm{loc}}^{\log
}\left( 
%TCIMACRO{\U{211d} }%
%BeginExpansion
\mathbb{R}
%EndExpansion
^{n}\right) $\textit{.\ }Let $q\in \mathcal{P}(\mathbb{R})$\ be log-H\"{o}%
lder continuous at the origin. From Lemma \ref{Key-lemma1.2} we have%
\begin{equation*}
\left\Vert \lambda \right\Vert _{\boldsymbol{b}_{p(\cdot ),q(\cdot
)}^{\alpha (\cdot )}}\approx \big\|\sum\limits_{m\in \mathbb{Z}^{n}}\lambda
_{0,m}\chi _{0,m}\big\|_{p(\cdot )}+\Big(\sum_{v=1}^{\infty }\big\|%
2^{v(\alpha \left( \cdot \right) -n/2)}\sum\limits_{m\in \mathbb{Z}%
^{n}}\lambda _{v,m}\chi _{v,m}\big\|_{p(\cdot )}^{q(0)}\Big)^{\frac{1}{q(0)}%
}.
\end{equation*}

Now we are now in a position to state the atomic decomposition theorem.

\begin{thm}
\label{atomic-dec} Let $\alpha \in C_{\mathrm{loc}}^{\log }\left( 
%TCIMACRO{\U{211d} }%
%BeginExpansion
\mathbb{R}
%EndExpansion
^{n}\right) $ and $p\in \mathcal{P}^{\log }\left( 
%TCIMACRO{\U{211d} }%
%BeginExpansion
\mathbb{R}
%EndExpansion
^{n}\right) $. Let $q\in \mathcal{P}\left( 
%TCIMACRO{\U{211d} }%
%BeginExpansion
\mathbb{R}
%EndExpansion
\right) $ be \emph{$\log $}-H\"{o}lder continuous at the origin\ with $1\leq
q^{-}\leq q^{+}<\infty $\textit{. }Let $K,L+1\in \mathbb{N}_{0}$ such that%
\begin{equation}
K\geq \lbrack \alpha ^{+}]+1,  \label{K,L,B-F-cond}
\end{equation}%
and%
\begin{equation}
L\geq \max (-1,[-\alpha ^{-}]).  \label{K,L,B-cond}
\end{equation}%
Then $f\in \mathcal{S}^{\prime }(\mathbb{R}^{n})$ belongs to $\boldsymbol{B}%
_{p(\cdot ),q(\cdot )}^{\alpha (\cdot )}$, if and only if it is represented
as%
\begin{equation}  \label{new-rep}
f=\sum\limits_{v=0}^{\infty }\sum\limits_{m\in \mathbb{Z}^{n}}\lambda
_{v,m}\varrho _{v,m},\text{ \ \ \ \ converging in } \mathcal{S}^{\prime }(%
\mathbb{R}^{n}),
\end{equation}%
where $\varrho _{v,m}$ are $\left[ K,L\right] $-atoms and $\lambda
=\{\lambda _{v,m}\in \mathbb{C}\}_{v\in \mathbb{N}_{0},m\in \mathbb{Z}%
^{n}}\in \boldsymbol{b}_{p(\cdot ),q(\cdot )}^{\alpha (\cdot )}$.
Furthermore, $\inf \left\Vert \lambda \right\Vert _{\boldsymbol{b}_{p(\cdot
),q(\cdot )}^{\alpha (\cdot )}}$ is an equivalent norm in $\boldsymbol{B}%
_{p(\cdot ),q(\cdot )}^{\alpha (\cdot )}$, where the infimum is taken over
admissible representations $\mathrm{\eqref{new-rep}}$.
\end{thm}

\begin{proof} %[Proof of Theorem \ref{atomic-dec}]
The idea of proof comes from \cite[Theorem 6]{FJ86} and
\cite[Theorem 4.3]{D6}. The proof of convergence of \eqref{new-rep}
in $\mathcal{S}^{\prime }(\mathbb{R}^{n})$ is
postponed to appendix.\\

We divide the proof into three steps.\\

\textit{Step 1}. Assume that $f\in \boldsymbol{B}_{p(\cdot ),q(\cdot
)}^{\alpha (\cdot )}$ and let $\Phi $ and $\varphi $ satisfy%
\begin{equation}
\text{supp}\, \mathcal{F}\Phi \subset \overline{B(0,2)}\text{ and }|\mathcal{F}%
\Phi (\xi )|\geq c\text{ if }|\xi |\leq \frac{5}{3}  \label{Ass3}
\end{equation}%
and
\begin{equation}
\text{supp}\, \mathcal{F}\varphi \subset \overline{B(0,2)}\backslash B(0,1/2)%
\text{ and }|\mathcal{F}\varphi (\xi )|\geq c\text{ if
}\frac{3}{5}\leq |\xi |\leq \frac{5}{3}.  \label{Ass4}
\end{equation}%
There exist \ functions $\Psi \in \mathcal{S}(\mathbb{R}^{n})$ satisfying $%
\mathrm{\eqref{Ass3}}$ and $\psi \in \mathcal{S}(\mathbb{R}^{n})$
satisfying
$\mathrm{\eqref{Ass4}}$ such that%
\begin{equation*}
f=\Psi \ast \Phi \ast f+\int_{0}^{1}\psi _{t}\ast \varphi _{t}\ast f\frac{dt%
}{t}
\end{equation*}%
(see Section 3). Using the definition of the cubes $Q_{v,m}$, we obtain%
\begin{equation*}
f(x)=\sum\limits_{m\in \mathbb{Z}^{n}}\int_{Q_{0,m}}\Phi (x-y)\Psi
\ast
f(y)dy+\sum_{v=1}^{\infty }\sum\limits_{m\in \mathbb{Z}^{n}}%
\int_{2^{-v}}^{2^{1-v}}\int_{Q_{v,m}}\varphi _{t}(x-y)\psi _{t}\ast f(y)dy%
\frac{dt}{t}.
\end{equation*}%
We define for every $v\geq 0$, $t\in \lbrack 2^{-v},2^{1-v}]$\ and
all $m\in
\mathbb{Z}^{n}$%
\begin{equation}
\lambda _{v,m}=C_{\varphi }\Big(\int_{2^{-v}}^{2^{1-v}}\int_{Q_{v,m}}\left\vert \psi _{t}\ast f(y)\right\vert ^{2}dy\frac{dt}{t}\Big)^{1/2},
\label{Coefficient}
\end{equation}%
where%
\begin{equation*}
C_{\varphi }=\max \{\sup_{\left\vert y\right\vert \leq c}\left\vert
D^{\alpha }\varphi (y)\right\vert :\left\vert \alpha \right\vert \leq K\},%
\text{\quad }c>0.
\end{equation*}%
Define also%
\begin{equation}
\varrho _{v,m}(x)=\left\{
\begin{array}{ccc}
\displaystyle{\frac{1}{\lambda
_{v,m}}\int_{2^{-v}}^{2^{1-v}}\int_{Q_{v,m}}\varphi _{t}(x-y)\psi
_{t}\ast f(y)dy\frac{dt}{t}} & \text{if} & \lambda _{v,m}\neq 0,
\\
0 & \text{if} & \lambda _{v,m}=0.%
\end{array}%
\right.  \label{K-L-atom}
\end{equation}%
Similarly we define for every $m\in \mathbb{Z}^{n}$ the numbers
$\lambda
_{0,m}$ and the functions $\varrho _{0,m}$ taking in $\mathrm{%
\eqref{Coefficient}}$ and $\mathrm{\eqref{K-L-atom}}$ $v=0$ and replacing $%
\psi _{t}$ and $\varphi $ by $\Psi $ and $\Phi $, respectively. Let
us now
check that such $\varrho _{v,m}$ are atoms in the sense of Definition \ref%
{Atom-Def}. Note that the support and moment conditions are clear by $%
\mathrm{\eqref{Ass3}}$ and $\mathrm{\eqref{Ass4}}$, respectively. It
thus remains to check $\mathrm{\eqref{diff-cond}}$ in Definition
\ref{Atom-Def}.
We have%
\begin{eqnarray*}
&&\left\vert D_{x}^{\beta }\varrho _{v,m}(x)\right\vert \\
&\leq &\frac{1}{\lambda _{v,m}}\int_{2^{-v}}^{2^{1-v}}\Big(%
\int_{Q_{v,m}}\left\vert D_{x}^{\beta }\varphi _{t}(x-y)\right\vert ^{2}dy%
\Big)^{1/2}\Big(\int_{Q_{v,m}}\left\vert \psi _{t}\ast f(y)\right\vert ^{2}dy%
\Big)^{1/2}\frac{dt}{t} \\
&\leq &\frac{1}{\lambda _{v,m}}\Big(\int_{2^{-v}}^{2^{1-v}}\int_{Q_{v,m}}%
\left\vert D_{x}^{\beta }\varphi _{t}(x-y)\right\vert ^{2}dy\frac{dt}{t}\Big)%
^{1/2}\Big(\int_{2^{-v}}^{2^{1-v}}\int_{Q_{v,m}}\left\vert \psi
_{t}\ast
f(y)\right\vert ^{2}dy\frac{dt}{t}\Big)^{1/2} \\
&\leq &\frac{1}{C_{\varphi
}}\Big(\int_{2^{-v}}^{2^{1-v}}t^{-2(n+\left\vert
\beta \right\vert )}\int_{Q_{v,m}}\left\vert (D^{\beta }\varphi )(\frac{x-y}{%
t})\right\vert ^{2}dy\frac{dt}{t}\Big)^{1/2} \\
&\lesssim &2^{v(\left\vert \beta \right\vert +n/2)}.
\end{eqnarray*}%
The modifications for the terms with $v=0$ are obvious.\\

\textit{Step }2\textit{.} Next we show that there is a constant
$c>0$ such that $$\left\Vert \lambda \right\Vert
_{\mathbf{b}_{p(\cdot ),q(\cdot )}^{\alpha (\cdot )}}\leq
c\left\Vert f\right\Vert _{\mathbf{B}_{p(\cdot ),q(\cdot )}^{\alpha
(\cdot )}}.$$ For that reason we exploit the equivalent norms given
in Theorem \ref{loc-mean-char} involving Peetre's maximal function.
Let $v\geq 0$. Taking into account that $\left\vert x-y\right\vert
\leq c$ $2^{-v}$ for $x,y\in Q_{v,m}$ we obtain%
\begin{equation*}
t^{\alpha \left( y\right) -\alpha \left( x\right) }\lesssim
2^{v(\alpha
\left( x\right) -\alpha \left( y\right) )}\leq \frac{c_{\log }(\alpha )v}{%
\log (e+1/\left\vert x-y\right\vert )}\leq \frac{c_{\log }(\alpha
)v}{\log (e+2^{-v}/c)}\leq c,  t\in \left[ 2^{-v},2^{1-v}\right] ,
\end{equation*}%
if $v\geq \left[ \log _{2}c\right] +2$. If $0<v<\left[ \log _{2}c\right] +2$%
, then $2^{v(\alpha \left( x\right) -\alpha \left( y\right) )}\leq
2^{v(\alpha ^{-}-\alpha ^{+})}\leq c$. Hence, we see that
\begin{equation*}
t^{-\alpha \left( x\right) }\left\vert \psi _{t}\ast f(y)\right\vert \leq c%
\text{ }t^{-\alpha \left( y\right) }\left\vert \psi _{t}\ast
f(y)\right\vert
\end{equation*}%
for any $x,y\in Q_{v,m}$ any $v\in \mathbb{N}$ and any $t\in \left[
2^{-v},2^{1-v}\right] $. Hence, we deduce that
\begin{eqnarray*}
&&\sum\limits_{m\in \mathbb{Z}^{n}}\lambda _{v,m}t^{-(\alpha \left(
x\right)
+n/2)}\chi _{v,m}(x) \\
&\leq &C_{\varphi }\sum\limits_{m\in \mathbb{Z}^{n}}t^{-\alpha
\left( x\right) }\sup_{t\in \left[ 2^{-v},2^{1-v}\right] }\sup_{y\in
Q_{v,m}}\left\vert \psi _{t}\ast f(y)\right\vert \chi _{v,m}(x) \\
&\leq &c\sum\limits_{m\in \mathbb{Z}^{n}}\sup_{\left\vert
y-x\right\vert \leq c\text{ }2^{-v}}\frac{t^{-\alpha (y)}\left\vert
\psi _{t}\ast
f(y)\right\vert (1+t^{-1}\left\vert y-x\right\vert )^{a}}{%
(1+t^{-1}\left\vert y-x\right\vert )^{a}}\chi _{v,m}(x) \\
&\leq &c\text{ }\psi _{t}^{\ast ,a}t^{-\alpha \left( \cdot \right)
}f(x)\sum\limits_{m\in \mathbb{Z}^{n}}\chi _{v,m}(x)=c\text{ }\psi
_{t}^{\ast ,a}t^{-\alpha \left( \cdot \right) }f(x),
\end{eqnarray*}%
where we have used $\sum\limits_{m\in \mathbb{Z}^{n}}\chi
_{v,m}(x)=1$.
Similarly, we obtain%
\begin{equation*}
\sum\limits_{m\in \mathbb{Z}^{n}}\lambda _{0,m}\chi _{0,m}(x)\leq c\text{ }%
\Psi ^{\ast ,a}f(x),
\end{equation*}%
which implies that
\begin{equation*}
\left\Vert \lambda \right\Vert _{\mathbf{b}_{p(\cdot ),q(\cdot
)}^{\alpha (\cdot )}}\leq c\left\Vert f\right\Vert
_{\boldsymbol{B}_{p(\cdot ),q(\cdot
)}^{\alpha (\cdot )}}^{\prime }\leq c\left\Vert f\right\Vert _{\boldsymbol{B}%
_{p(\cdot ),q(\cdot )}^{\alpha (\cdot )}},
\end{equation*}%
by Theorem \ref{loc-mean-char} (with the equivalent norm $\mathrm{%
\eqref{new-norm1}}$\textrm{)}.\vskip5pt

\textit{Step }3\textit{.} Assume that $f$ can be represented by $\mathrm{%
\eqref{new-rep}}$, with $K$ and $L$ satisfying
$\mathrm{\eqref{K,L,B-F-cond}} $ and $\mathrm{\eqref{K,L,B-cond}}$,
respectively. We show that $f\in \boldsymbol{B}_{p\left( \cdot
\right) ,q\left( \cdot \right) }^{\alpha
\left( \cdot \right) }$, and that for some $c>0$, $$\left\Vert f\right\Vert _{%
\boldsymbol{B}_{p\left( \cdot \right) ,q\left( \cdot \right)
}^{\alpha
\left( \cdot \right) }}\leq c\left\Vert \lambda \right\Vert _{\boldsymbol{b}%
_{p\left( \cdot \right) ,q\left( \cdot \right) }^{\alpha \left(
\cdot \right) }}.$$ We use the equivalent norm given in
\eqref{new-norm1}. The arguments are similar to those in \cite{D6}.
We write
\begin{eqnarray*}
f &=&\sum\limits_{v=0}^{\infty }\sum\limits_{m\in
\mathbb{Z}^{n}}\lambda _{v,m}\varrho _{v,m}\\
%=\sum\limits_{m\in
%\mathbb{Z}^{n}}\lambda _{0,m}\varrho
%_{0,m}+\sum\limits_{v=1}^{\infty }\sum\limits_{m\in
%\mathbb{Z}^{n}}
%\lambda_{v,m}\varrho _{v,m} \\
&=&\sum\limits_{m\in \mathbb{Z}^{n}}\lambda _{0,m}\varrho
_{0,m}+\sum\limits_{v=1}^{j} \sum_{m\in
\mathbb{Z}^{n}}\lambda_{v,m}\varrho _{v,m}
+\sum\limits_{v=j+1}^{\infty }\sum_{m\in
\mathbb{Z}^{n}}\lambda_{v,m}\varrho _{v,m}.
\end{eqnarray*}%
From Lemmas \ref{Conv-est1} and \ref{FJ-lemma}, we have for any $M$
sufficiently large, $t\in \left[ 2^{-j},2^{1-j}\right] $ and any\
$0\leq
v\leq j$%
\begin{eqnarray*}
&&\sum\limits_{m\in \mathbb{Z}^{n}}t^{-\alpha \left( x\right)
}\left\vert \lambda _{v,m}\right\vert \left\vert \varphi _{t}\ast
\varrho
_{v,m}(x)\right\vert \\
&\lesssim &2^{(v-j)(K-\alpha ^{+})}\sum\limits_{m\in \mathbb{Z}%
^{n}}2^{v(\alpha \left( x\right) -n/2)}\left\vert \lambda
_{v,m}\right\vert
\eta _{v,M}(x-x_{Q_{v,m}}) \\
&\lesssim &2^{(v-j)(K-\alpha ^{+})}\sum\limits_{m\in \mathbb{Z}%
^{n}}2^{v(\alpha \left( x\right) +n/2)}\left\vert \lambda
_{v,m}\right\vert \eta _{v,M}\ast \chi _{v,m}(x).
\end{eqnarray*}%
Put
\[
I_{t,j}:=t^{-\frac{1}{q(t)}}\sum_{v=1}^{j}2^{(v-j)(K-\alpha
^{+})}\eta _{v,T}\ast \Big(2^{v(\alpha \left( \cdot \right)
+n/2)}\sum\limits_{m\in \mathbb{Z}^{n}}\left\vert \lambda
_{v,m}\right\vert
\chi _{v,m}\Big)\Big\|_{p(\cdot )}\chi _{\left[ 2^{-j},2^{1-j}\right] }\Big)%
_{j} .
\]
Since $$2^{v\alpha \left( \cdot \right) }\eta _{v,M}\ast \chi
_{v,m}\lesssim \eta _{v,T}\ast 2^{v\alpha \left( \cdot \right) }\chi
_{v,m}$$ by Lemma \ref{DHR-lemma} for $t=M-c_{\log }(\alpha )$, and
since
$K>\alpha^{+} $, we apply Lemma \ref{Key-lemma1} to obtain%
\begin{equation*}
\begin{split}
&\Big\| I_{t,j} \Big\|_{\ell _{>}^{q(\cdot )}(L^{q(\cdot )})} \\
\lesssim&
\Big\|\Big(t^{-\frac{1}{q(t)}}\sum_{v=0}^{j}2^{(v-j)(K-\alpha
^{+})}\Big\|2^{v(\alpha \left( \cdot \right) +n/2)}\sum\limits_{m\in \mathbb{%
Z}^{n}}\lambda _{v,m}\chi _{v,m}\Big\|_{p(\cdot )}\chi _{\left[
2^{-j},2^{1-j}\right] }\Big)_{j}\Big\|_{\ell _{>}^{q(\cdot
)}(L^{q(\cdot )})}
\\
\lesssim &\Big\|\Big(t^{-\frac{1}{q(t)}}\Big\|2^{j(\alpha \left(
\cdot
\right) +n/2)}\sum\limits_{m\in \mathbb{Z}^{n}}\lambda _{j,m}\chi _{j,m}%
\Big\|_{p(\cdot )}\chi _{\left[ 2^{-j},2^{1-j}\right]
}\Big)_{j}\Big\|_{\ell
_{>}^{q(\cdot )}(L^{q(\cdot )})}\\
\lesssim& \left\Vert \lambda \right\Vert _{%
\boldsymbol{b}_{p(\cdot ),q(\cdot )}^{\alpha (\cdot )}}.
\end{split}
\end{equation*}%
For $v=0$, we have
\begin{eqnarray*}
&&\Big\|\Big(\Big\|t^{-\frac{1}{q(t)}}2^{-j(K-\alpha ^{+})}\eta _{0,M}\ast %
\Big(\sum\limits_{m\in \mathbb{Z}^{n}}\left\vert \lambda
_{0,m}\right\vert
\chi _{0,m}\Big)\Big\|_{p(\cdot )}\chi _{\left[ 2^{-j},2^{1-j}\right] }\Big)%
_{j}\Big\|_{\ell _{>}^{q(\cdot )}(L^{q(\cdot )})} \\
&\lesssim &\Big\|\sum\limits_{m\in \mathbb{Z}^{n}}\lambda _{0,m}\chi _{0,m}%
\Big\|_{p(\cdot )}\Big\|\Big(t^{-\frac{1}{q(t)}}2^{-j(K-\alpha ^{+})}\chi _{%
\left[ 2^{-j},2^{1-j}\right] }\Big)_{j}\Big\|_{\ell _{>}^{q(\cdot
)}(L^{q(\cdot )})} \\
&\lesssim &\left\Vert \lambda \right\Vert _{\boldsymbol{b}_{p(\cdot
),q(\cdot )}^{\alpha (\cdot )}}.
\end{eqnarray*}%
Now from Lemma \ref{FJ-lemma}, we have for any $M$ sufficiently large, $t\in %
\left[ 2^{-j},2^{1-j}\right] $ and $v\geq j$%
\begin{eqnarray*}
&&\sum\limits_{m\in \mathbb{Z}^{n}}t^{-\alpha \left( x\right)
}\left\vert \lambda _{v,m}\right\vert \left\vert \varphi _{t}\ast
\varrho
_{v,m}(x)\right\vert \\
&\lesssim &2^{(j-v)(L+1+n/2)}\sum\limits_{m\in
\mathbb{Z}^{n}}2^{j(\alpha \left( x\right) -n/2)}\left\vert \lambda
_{v,m}\right\vert \eta
_{j,M}(x-x_{Q_{v,m}}) \\
&\lesssim &2^{(j-v)(L+1+n/2)}\sum\limits_{m\in
\mathbb{Z}^{n}}2^{j(\alpha \left( x\right) -n/2)}\left\vert \lambda
_{v,m}\right\vert \eta _{j,M}\ast \eta _{v,M}(x-x_{Q_{v,m}}),
\end{eqnarray*}%
where we used Lemma \ref{Conv-est2} in the last step, since $\eta
_{j,M}=\eta _{\min (v,j),M}$. Again by Lemma \ref{Conv-est1}, we have%
\begin{equation*}
\eta _{j,M}\ast \eta _{v,M}(x-x_{Q_{v,m}})\lesssim 2^{vn}\eta
_{j,M}\ast \eta _{v,M}\ast \chi _{v,m}(x).
\end{equation*}%
Therefore, we estimate as
\begin{equation*}
\begin{split}
& \sum\limits_{m\in \mathbb{Z}^{n}}t^{-\alpha \left( x\right)
}\left\vert \lambda _{v,m}\right\vert \left\vert \varphi _{t}\ast
\varrho
_{v,m}(x)\right\vert\\
\le &c\text{ }2^{(j-v)(L+1-n/2)}\sum\limits_{m\in
\mathbb{Z}^{n}}2^{j(\alpha \left( x\right) +n/2)}\left\vert \lambda
_{v,m}\right\vert \eta _{j,M}\ast
\eta _{v,M}\ast \chi _{v,m}(x) \\
\lesssim &2^{(j-v)(L+1+\alpha ^{-})}\eta _{j,T}\ast \eta _{v,T}\ast \Big[%
2^{v(\alpha \left( \cdot \right) +n/2)}\sum\limits_{m\in \mathbb{Z}%
^{n}}\left\vert \lambda _{v,m}\right\vert \chi _{v,m}\Big](x),
\end{split}
\end{equation*}%
by Lemmas \ref{DHR-lemma} and \ref{Conv-est2}, with $t=M-c_{\log
}(\alpha )$. Since the convolution with a radially decreasing
$L^{1}$-function is bounded on $L^{p(\cdot )}$, it follows that
\begin{equation*}
\Big\|\Big(t^{-\frac{1}{q(t)}}\sum_{v=j}^{\infty }2^{(j-v)H}%
\big\|\eta _{j,T}\ast \eta _{v,T}\ast 2^{v(\alpha \left( \cdot
\right)
+\frac{n}{2})}\sum\limits_{m\in \mathbb{Z}^{n}}|\lambda _{v,m}|\chi _{v,m}\big\|%
_{p(\cdot )}\chi _{\left[ 2^{-j},2^{1-j}\right]
}\Big)_{j}\Big\|_{\ell _{>}^{q(\cdot )}(L^{q(\cdot )})}
\end{equation*}%
is bounded by%
\begin{equation*}
\Big\|\Big(t^{-\frac{1}{q(t)}}\sum_{v=j}^{\infty }2^{(j-v)H}\big\|%
2^{v(\alpha \left( \cdot \right) +n/2)}\sum\limits_{m\in \mathbb{Z}%
^{n}}\lambda _{v,m}\chi _{v,m}\big\|_{p(\cdot )}\chi _{\left[ 2^{-j},2^{1-j}%
\right] }\Big)_{j}\Big\|_{\ell _{>}^{q(\cdot )}(L^{q(\cdot )})},
\end{equation*}%
where we put $H:=L+1+\alpha ^{-}$. Observing that $H>0$, an application of Lemma %
\ref{Key-lemma1} yields that the last expression is bounded by%
\begin{equation*}
c\Big\|\Big(t^{-\frac{1}{q(t)}}\big\|2^{j(\alpha \left( \cdot
\right)
+n/2)}\sum\limits_{m\in \mathbb{Z}^{n}}\lambda _{j,m}\chi _{j,m}\big\|%
_{p(\cdot )}\chi _{\left[ 2^{-j},2^{1-j}\right]
}\Big)_{j}\Big\|_{\ell
_{>}^{q(\cdot )}(L^{q(\cdot )})}\lesssim \left\Vert \lambda \right\Vert _{%
\boldsymbol{b}_{p\left( \cdot \right) ,q\left( \cdot \right)
}^{\alpha \left( \cdot \right) }}.
\end{equation*}%
Clearly, $\sum\limits_{m\in \mathbb{Z}^{n}}\left\vert \lambda
_{v,m}\right\vert \left\vert \Phi \ast \varrho _{v,m}(x)\right\vert
$ is
bounded by%
\begin{equation*}
c\text{ }2^{-v(L+1+\alpha ^{-})}\eta _{0,T}\ast \eta _{v,T}\ast \Big[%
2^{v(\alpha \left( \cdot \right) +n/2)}\sum\limits_{m\in \mathbb{Z}%
^{n}}\left\vert \lambda _{v,m}\right\vert \chi _{v,m}\Big](x).
\end{equation*}%
Taking the $L^{p\left( \cdot \right) }$-norm and the fact the
convolution with a radially decreasing $L^{1}$-function is bounded
on $L^{p(\cdot )}$,
we get%
\begin{equation*}
\big\|\Phi \ast f\big\|_{p(\cdot )}\lesssim \sum_{v=0}^{\infty
}2^{-v(L+1+\alpha ^{-})}\big\|2^{v(\alpha \left( \cdot \right)
+n/2)}\sum\limits_{m\in \mathbb{Z}^{n}}\lambda _{v,m}\chi _{v,m}\big\|%
_{p(\cdot )}\lesssim \left\Vert \lambda \right\Vert _{\boldsymbol{\tilde{b}}%
_{p(\cdot ),q(\cdot )}^{s(\cdot )}},
\end{equation*}%
{where }%
\begin{equation*}
\left\Vert \lambda \right\Vert _{\boldsymbol{\tilde{b}}_{p(\cdot
),q(\cdot )}^{\alpha (\cdot )}}=\sup_{v\geq
0}\big\|\big\|t^{-(\alpha (\cdot
)+n/2)}\sum\limits_{m\in \mathbb{Z}^{n}}\lambda _{v,m}\chi _{v,m}\big\|%
_{p(\cdot )}\big\|_{L^{q(\cdot )}([2^{-v},2^{1-v}],\frac{dt}{t})}.
\end{equation*}%
Indeed, we see from H\"{o}lder's inequality that
\begin{eqnarray*}
&&\big\|2^{v(\alpha \left( \cdot \right) +n/2)}\sum\limits_{m\in \mathbb{Z}%
^{n}}\lambda _{v,m}\chi _{v,m}\big\|_{p(\cdot )} \\
&\lesssim &\frac{1}{\log 2}\int_{2^{-v}}^{2^{1-v}}\big\|t^{-(\alpha
(\cdot
)+n/2)}\sum\limits_{m\in \mathbb{Z}^{n}}\lambda _{v,m}\chi _{v,m}\big\|%
_{p(\cdot )}\frac{dt}{t} \\
&\lesssim &\big\|\big\|t^{-(\alpha (\cdot )+n/2)}\sum\limits_{m\in \mathbb{Z}%
^{n}}\lambda _{v,m}\chi _{v,m}\big\|_{p(\cdot )}\big\|_{L^{q(\cdot
)}([2^{-v},2^{1-v}],\frac{dt}{t})}\big\|1\big\|_{L^{q^{\prime
}(\cdot
)}([2^{-v},2^{1-v}],\frac{dt}{t})} \\
&\lesssim &\left\Vert \lambda \right\Vert
_{\boldsymbol{\tilde{b}}_{p(\cdot ),q(\cdot )}^{\alpha (\cdot )}},
\end{eqnarray*}%
where $q^{\prime }(\cdot )$ is the conjugate exponent of $q(\cdot
)$. Our estimate follows from the embedding
$$\boldsymbol{b}_{p(\cdot
),q(\cdot )}^{\alpha (\cdot )}\hookrightarrow
\boldsymbol{\tilde{b}}_{p(\cdot),q(\cdot )}^{\alpha (\cdot )},$$ and
hence, the proof of Theorem \ref{atomic-dec} is complete.
\end{proof}

% \appendix

\section{Appendix}

Here we present the proof of Lemma \ref{DHHR-estimate} and the convergence
of \eqref{new-rep}.

\textbf{Proof of Lemma \ref{DHHR-estimate}}. Let $p\in \mathcal{P}^{\mathrm{%
log}}(\mathbb{R}^{n})$ with $1\leq p^{-}\leq p^{+}<\infty $, and recall that 
\begin{equation*}
p_{Q}^{-}=\underset{z\in Q}{\text{ess-inf}}\,p(z).
\end{equation*}%
Define $q\in \mathcal{P}^{\mathrm{log}}(\mathbb{R}^{n}\times \mathbb{R}^{n})$
by%
\begin{equation}
\frac{1}{q(x,y)}:=\max \left( \frac{1}{p(x)}-\frac{1}{p(y)},0\right) .
\label{EQ:q}
\end{equation}%
We split $f(y)$ into three parts: 
\begin{equation*}
f_{1}(y)=f(y)\chi _{\{z\in Q:\left\vert f(z)\right\vert >1\}}(y),
\end{equation*}%
\begin{equation*}
f_{2}(y)=f(y)\chi _{\{z\in Q:\left\vert f(z)\right\vert \leq 1,p(z)\leq
p(x)\}}(y),
\end{equation*}%
\begin{equation*}
f_{3}(y)=f(y)\chi _{\{z\in Q:\left\vert f(z)\right\vert \leq
1,p(z)>p(x)\}}(y).
\end{equation*}%
Then we can write 
\begin{eqnarray*}
\Big(\frac{\gamma _{m}}{w(Q)}\int_{Q}\left\vert f(y)\right\vert w(y)\,dy\Big)%
^{p\left( x\right) } &\leq &\max (3^{p^{+}-1},1)\sum\limits_{i=1}^{3}\Big(%
\frac{\gamma _{m}}{w(Q)}\int_{Q}\left\vert f_{i}(y)\right\vert w(y)\,dy\Big)%
^{p\left( x\right) } \\
&=&\max (3^{p^{+}-1},1)\left( I_{1}+I_{2}+I_{3}\right) .
\end{eqnarray*}

\textbf{Estimation of}\textit{\ }$I_{1}$. By H\"{o}lder's inequality, we see
that 
\begin{equation*}
I_{1}\leq \gamma _{m}^{p\left( x\right) }\Big(\frac{1}{w(Q)}%
\int_{Q}\left\vert f_{1}(y)\right\vert ^{p_{Q}^{-}}w(y)\,dy\Big)^{\frac{%
p\left( x\right) }{p_{Q}^{-}}}.
\end{equation*}%
Since $|f_{1}(y)|>1$, we have $\left\vert f_{1}(y)\right\vert
^{p_{Q}^{-}}\leq \left\vert f_{1}(y)\right\vert ^{p(y)}\leq \left\vert
f(y)\right\vert ^{p(y)}$ and thus we conclude that 
\begin{eqnarray*}
I_{1} &\leq &\gamma _{m}^{p\left( x\right) }\Big(\frac{1}{w(Q)}%
\int_{Q}\left\vert f(y)\right\vert ^{p\left( y\right) }w(y)\,dy\Big)^{\frac{%
p\left( x\right) }{p_{Q}^{-}}-1}\Big(\frac{1}{w(Q)}\int_{Q}\left\vert
f(y)\right\vert ^{p\left( y\right) }w(y)\,dy\Big) \\
&\leq &w(Q)^{1-\frac{p\left( x\right) }{p_{Q}^{-}}}\Big(\frac{1}{w(Q)}%
\int_{Q}\left\vert f(y)\right\vert ^{p\left( y\right) }w(y)\,dy\Big),
\end{eqnarray*}%
where we used the fact that 
\begin{equation*}
\int_{Q}\left\vert f(y)\right\vert ^{p\left( y\right) }w(y)\,dy\leq 1.
\end{equation*}%
Obviously, if $Q=(a,b)\subset \mathbb{R}$, with $0<a<b<\infty $, we have the
same estimate for $p_{Q}^{-}$ replaced by $p^{-}$.

\textbf{Estimation of}\textit{\ }$I_{2}$. Again by H\"{o}lder's inequality,
we see that 
\begin{equation*}
I_{2}\leq \gamma _{m}^{p\left( x\right) }\frac{1}{w(Q)}\int_{Q}\left\vert
f_{2}(y)\right\vert ^{p\left( x\right) }w(y)\,dy=:J.
\end{equation*}%
Since $|f_{2}(y)|\leq 1$, it follows that $\left\vert f_{2}(y)\right\vert
^{p(x)}\leq \left\vert f_{2}(y)\right\vert ^{p(y)}$ and hence, 
\begin{equation*}
J\leq \frac{1}{w(Q)}\int_{Q}\left\vert f(y)\right\vert ^{p(y)}w(y)\,dy,
\end{equation*}%
which is true if $Q=(a,b)\subset \mathbb{R}$, with $0<a<b<\infty $. Thus we
conclude that 
\begin{equation*}
I_{2}\leq \frac{1}{w(Q)}\int_{Q}\left\vert f(y)\right\vert ^{p(y)}w(y)\,dy.
\end{equation*}

\textbf{Estimation of}\textit{\ }$I_{3}$. Again by H\"{o}lder's inequality,
we see that 
\begin{equation}
\Big(\frac{\gamma _{m}}{w(Q)}\int_{Q}\left\vert f_{3}(y)\right\vert w(y)dy%
\Big)^{p\left( x\right) }\leq \frac{1}{w(Q)}\int_{Q}\left( \gamma
_{m}\left\vert f_{3}(y)\right\vert \right) ^{p\left( x\right) }w(y)dy.
\label{Est.I3}
\end{equation}%
Observe that%
\begin{equation*}
\frac{1}{p(x)}=\frac{1}{p(y)}+\frac{1}{p(x)}-\frac{1}{p(y)}\leq \frac{1}{p(y)%
}+\frac{1}{q(x,y)},\quad x,y\in Q,
\end{equation*}%
where $q(x,y)$ is defined in \eqref{EQ:q}. Therefore,%
\begin{equation*}
\left( \gamma _{m}\left\vert f_{3}(y)\right\vert \right) ^{p\left( x\right)
}\leq \left\vert f_{3}(y)\right\vert ^{p(y)}+\gamma _{m}^{q(x,y)},\quad
x,y\in Q,
\end{equation*}%
by using Young's inequality. Hence the right member of \eqref{Est.I3} is
dominated by%
\begin{eqnarray*}
&&\frac{1}{w(Q)}\int_{Q}\left( \left\vert f(y)\right\vert ^{p(y)}+\gamma
_{m}^{q(x,y)}\right) \chi _{\{z\in Q:\left\vert f(z)\right\vert \leq
1,p(z)>p(x)\}}(y)w(y)\,dy \\
&\leq &\frac{1}{w(Q)}\int_{Q}\left( \left\vert f(y)\right\vert
^{p(y)}+\gamma _{m}^{q(x,y)}\right) w(y)\,dy.
\end{eqnarray*}%
Now observe that 
\begin{equation*}
\frac{1}{q(x,y)}\leq \frac{1}{s(x)}+\frac{1}{s(y)},
\end{equation*}%
where $\frac{1}{s(\cdot )}=\left\vert \frac{1}{p(\cdot )}-\frac{1}{p_{\infty
}}\right\vert $. We have 
\begin{equation*}
\gamma _{m}^{q(x,y)}=\gamma _{m}^{q(x,y)/2}\gamma _{m}^{q(x,y)/2}
\end{equation*}%
and%
\begin{equation*}
\gamma _{m}^{q(x,y)/2}\leq \gamma _{m}^{s(x)/4}+\gamma _{m}^{s(y)/4},
\end{equation*}%
again by Young's inequality. We suppose that $\left\vert Q\right\vert <1$.
Since $p\in \mathcal{P}^{\mathrm{log}}(\mathbb{R}^{n})$, we have%
\begin{equation*}
\frac{1}{q(x,y)}\leq \left\vert \frac{1}{p(x)}-\frac{1}{p(y)}\right\vert
\leq \frac{c\left( 1/p\right) }{-\log \left\vert Q\right\vert }.
\end{equation*}%
Hence, 
\begin{equation*}
\gamma _{m}^{q(x,y)/2}=e^{-2mc(1/p)q(x,y)}\leq \left\vert Q\right\vert ^{m}.
\end{equation*}%
If $\left\vert Q\right\vert \geq 1$, then we use $\gamma _{m}^{q(x,y)/2}\leq
1$ which follows from $\gamma _{m}<1$. We have%
\begin{equation*}
\frac{1}{s(x)}=\left\vert \frac{1}{p(x)}-\frac{1}{p_{\infty }}\right\vert
\leq \frac{c\left( 1/p\right) }{\log (e+\left\vert x\right\vert )}.
\end{equation*}%
Hence,%
\begin{equation*}
\gamma _{m}^{s(x)/4}=e^{-mc\left( 1/p\right) s(x)}\leq (e+\left\vert
x\right\vert )^{-m},\quad x\in Q.
\end{equation*}%
Similarly, we obtain 
\begin{equation*}
\gamma _{m}^{s(y)/4}=e^{-mc\left( 1/p\right) s(y)}\leq (e+\left\vert
y\right\vert )^{-m},\quad y\in Q.
\end{equation*}%
Summarizing the estimates obtained now, we conclude the required inequality %
\eqref{EQ:Key}.

We now turn to prove \eqref{EQ:key1} with $Q=(a,b)\subset \mathbb{R}$,%
\begin{equation*}
\omega (m,b)=\min \left( b^{m},1\right) \text{,\quad }\phi (y)=\left\vert
f(y)\right\vert ^{p(y)}
\end{equation*}%
and%
\begin{equation*}
g(x,y)=\left( e+\frac{1}{x}\right) ^{-m}+\left( e+\frac{1}{y}\right) ^{-m}
\end{equation*}%
with $p\in \mathcal{P}(\mathbb{R})$ being $\log $-H\"{o}lder continuous at
the origin. Recall that%
\begin{equation*}
I_{3}(x)\leq \frac{1}{w(Q)}\int_{Q}\left( \left\vert f(y)\right\vert
^{p(y)}+\gamma _{m}^{q(x,y)}\right) w(y)dy.
\end{equation*}%
Observe that 
\begin{equation*}
\frac{1}{q(x,y)}=\max (\frac{1}{p(x)}-\frac{1}{p(y)},0)\leq \frac{1}{h(x)}+%
\frac{1}{h(y)},
\end{equation*}%
where $\frac{1}{h(\cdot )}=\left\vert \frac{1}{p(\cdot )}-\frac{1}{p(0)}%
\right\vert $ and $x,y\in Q=(a,b)$, with $0<a<b<\infty $. We have 
\begin{equation*}
\gamma _{m}^{q(x,y)}=\gamma _{m}^{q(x,y)/2}\gamma _{m}^{q(x,y)/2}\leq \gamma
_{m}^{q(x,y)/2}\left( \gamma _{m}^{h(x)/4}+\gamma _{m}^{h(y)/4}\right) ,
\end{equation*}%
again by Young's inequality. We suppose that $b<1$. Then for any $x,y\in
Q=(a,b)$, with $0<a<b<\infty $, we have%
\begin{equation*}
\frac{1}{q(x,y)}\leq \frac{1}{h(x)}+\frac{1}{h(y)}\leq \frac{c_{\log }\left(
1/p\right) }{\log (e+\frac{1}{x})}+\frac{c_{\log }\left( 1/p\right) }{\log
(e+\frac{1}{y})}\leq \frac{2c_{\log }\left( 1/p\right) }{\log (e+\frac{1}{b})%
},
\end{equation*}%
since $p\in \mathcal{P}(\mathbb{R})$ is log-H\"{o}lder continuous at the
origin. Hence,%
\begin{equation*}
\gamma _{m}^{q(x,y)/2}=e^{-2mc_{\log }(1/p)q(x,y)}\leq b^{m}.
\end{equation*}%
If $b\geq 1$, then $\gamma _{m}^{q(x,y)/2}\leq 1$, which follows again from $%
\gamma _{m}<1$. We have%
\begin{equation*}
\gamma _{m}^{h(x)/4}=e^{-\frac{mc_{\log }\left( 1/p\right) }{\left\vert 
\frac{1}{p(x)}-\frac{1}{p(0)}\right\vert }}\leq e^{-m\log (e+\frac{1}{x}%
)}=\left( e+\frac{1}{x}\right) ^{-m},\quad x\in Q.
\end{equation*}%
Similarly, we obtain%
\begin{equation*}
\gamma _{m}^{h(y)/4}\leq \left( e+\frac{1}{y}\right) ^{-m},\quad y\in Q.
\end{equation*}%
We now turn to prove \eqref{EQ:key1} with $Q=(a,b)\subset \mathbb{R}$,%
\begin{equation}
\text{\quad }\phi (y)=\left\vert f(y)\right\vert ^{p(0)}\text{,\quad }%
g(x,y)=\left( e+\frac{1}{x}\right) ^{-m}\chi _{\{z\in Q:p(z){<}p(0)\}}(x),
\label{Est-p(0)}
\end{equation}%
$\omega (m,b)=\min \left( b^{m},1\right) $\ and\ $p\in \mathcal{P}(\mathbb{R}%
)$ being $\log $-H\"{o}lder continuous at the origin. We split $f(y)$ into
two\ parts

\begin{equation*}
\begin{array}{ccc}
f_{1}(y) & = & f(y)\chi _{\{z\in Q:\left\vert f(z)\right\vert >1\}}(y) \\ 
f_{2}(y) & = & f(y)\chi _{\{z\in Q:\left\vert f(z)\right\vert \leq 1\}}(y)%
\end{array}%
\end{equation*}%
for any $y\in Q$. Then we can write 
\begin{eqnarray*}
\Big(\frac{\gamma _{m}}{w(Q)}\int_{Q}\left\vert f(y)\right\vert w(y)dy\Big)%
^{p\left( x\right) } &\leq &\max (2^{p^{+}-1})\sum\limits_{i=1}^{2}\Big(%
\frac{\gamma _{m}}{w(Q)}\int_{Q}\left\vert f_{i}(y)\right\vert w(y)dy\Big)%
^{p\left( x\right) } \\
&=&\max (2^{p^{+}-1})\left( J_{1}(x)+J_{2}(x)\right) 
\end{eqnarray*}%
for any $x\in Q=(a,b)$\ with $0<a<b<\infty $. As in the estimation of $I_{1}$
we obtain%
\begin{equation*}
J_{1}(x)\leq (w(Q))^{1-\frac{p\left( x\right) }{p^{-}}}\Big(\frac{1}{w(Q)}%
\int_{Q}\left\vert f(y)\right\vert ^{p\left( 0\right) }w(y)dy\Big).
\end{equation*}%
We write 
\begin{eqnarray*}
J_{2}(x) &=&J_{2}(x)\chi _{\{z\in Q:p(z)\geq p(0)\}}(x)+J_{2}(x)\chi
_{\{z\in Q:p(z){<}p(0)\}}(x) \\
&=&J_{3}(x)+J_{4}(x).
\end{eqnarray*}%
By H\"{o}lder's inequality, we have 
\begin{equation*}
J_{3}(x)\leq \gamma _{m}^{p\left( x\right) }\frac{1}{w(Q)}\int_{Q}\left\vert
f_{2}(y)\right\vert ^{p\left( 0\right) }w(y)dy,
\end{equation*}%
since $|f_{2}(y)|\leq 1$ and $p(x)\geq p(0)$. Again by H\"{o}lder's
inequality, we have 
\begin{equation*}
J_{4}(x)\leq \frac{1}{w(Q)}\int_{Q}\left( \gamma _{m}\left\vert
f(y)\right\vert \right) ^{p\left( x\right) }w(y)dy\chi _{\{z\in Q:p(z){<}%
p(0)\}}(x).
\end{equation*}%
Now, thanks to Young's inequality, the last term is dominated by%
\begin{equation*}
J_{4}(x)\leq \frac{1}{w(Q)}\int_{Q}\left( \left\vert f(y)\right\vert
^{p(0)}+\gamma _{m}^{\sigma (x,0)}\right) w(y)dy\chi _{\{z\in Q:p(z){<}%
p(0)\}}(x),
\end{equation*}%
where 
\begin{equation*}
\frac{1}{\sigma (x,0)}=\frac{1}{p(x)}-\frac{1}{p(0)},\quad x\in \{z\in Q:p(z)%
{<}p(0)\}.
\end{equation*}%
We have%
\begin{equation*}
\frac{1}{\sigma (x,0)}\leq \frac{c_{\log }\left( 1/p\right) }{\log (e+\frac{1%
}{x})}\leq \frac{c_{\log }\left( 1/p\right) }{\log (e+\frac{1}{b})},\quad
x\in \{z\in Q:p(z){<}p(0)\},
\end{equation*}%
since $p\in \mathcal{P}(\mathbb{R})$ is log-H\"{o}lder continuous at the
origin. Then%
\begin{equation*}
\gamma _{m}^{\sigma (x,0)/2}=e^{-2\sigma (x,0)mc_{\log }\left( 1/p\right)
}\leq (e+\frac{1}{x})^{-m},\quad x\in \{z\in Q:p(z){<}p(0)\},
\end{equation*}%
and%
\begin{equation*}
\gamma _{m}^{\sigma (x,0)/2}\leq \left\{ 
\begin{array}{ccc}
b^{m}, & \text{if} & b<1; \\ 
1, & \text{if} & b\geq 1.%
\end{array}%
\right. 
\end{equation*}%
We now turn to prove \eqref{EQ:key1} with $Q=(a,b)\subset \mathbb{R}$, 
\begin{equation*}
\omega (m,b)=1\text{\quad }\gamma _{m}=e^{-mc_{\log }},\quad \phi
(y)=\left\vert f(y)\right\vert ^{p_{\infty }}
\end{equation*}%
and%
\begin{equation*}
g(x,y)=(e+x)^{-m}\chi _{\{z\in Q:p(z){<}p_{\infty }\}}(x)
\end{equation*}%
with $p\in \mathcal{P}(\mathbb{R})$ satisfying the $\log $-H\"{o}lder decay
condition. We employ the same notation as in the proof of \eqref{Est-p(0)}.
We need only to estimate $J_{2}$. We write 
\begin{eqnarray*}
J_{2}(x) &=&J_{2}(x)\chi _{\{z\in Q:p(z)\geq p_{\infty }\}}(x)+J_{2}(x)\chi
_{\{z\in Q:p(z){<}p_{\infty }\}}(x) \\
&=&J_{5}(x)+J_{6}(x).
\end{eqnarray*}%
By H\"{o}lder's inequality, we have 
\begin{equation*}
J_{5}(x)\leq \gamma _{m}^{p\left( x\right) }\frac{1}{w(Q)}\int_{Q}\left\vert
f_{2}(y)\right\vert ^{p_{\infty }}w(y)dy,
\end{equation*}%
since $|f_{2}(y)|\leq 1$ and $p(x)\geq p_{\infty }$. Again by H\"{o}lder's
inequality, we have 
\begin{equation*}
J_{6}(x)\leq \frac{1}{w(Q)}\int_{Q}\left( \gamma _{m}\left\vert
f(y)\right\vert \right) ^{p\left( x\right) }w(y)dy\chi _{\{z\in Q:p(z){<}%
p_{\infty }\}}(x).
\end{equation*}%
Again by Young's inequality, the last term is dominated by%
\begin{equation*}
J_{6}(x)\leq \frac{1}{w(Q)}\int_{Q}\left( \left\vert f(y)\right\vert
^{p_{\infty }}+\gamma _{m}^{\sigma (x,p_{\infty })}\right) w(y)dy\chi
_{\{z\in Q:p(z){<}p_{\infty }\}}(x),
\end{equation*}%
where 
\begin{equation*}
\frac{1}{\sigma (x,p_{\infty })}=\frac{1}{p(x)}-\frac{1}{p_{\infty }},\quad
x\in \{z\in Q:p(z){<}p_{\infty }\}.
\end{equation*}%
We have%
\begin{equation*}
\frac{1}{\sigma (x,p_{\infty })}\leq \frac{c_{\log }}{\log (e+x)},\quad x\in
\{z\in Q:p(z){<}p_{\infty }\},
\end{equation*}%
since $p\in \mathcal{P}(\mathbb{R})$ satisfies the log-H\"{o}lder decay
condition. Then%
\begin{eqnarray*}
\gamma _{m}^{\sigma (x,p_{\infty })} &=&e^{-\sigma (x,p_{\infty })mc_{\log }}
\\
&\leq &(e+x)^{-m},\quad x\in \{z\in Q:p(z){<}p_{\infty }\}.
\end{eqnarray*}%
Hence the lemma is proved.

\textbf{Proof of the convergence of \eqref{new-rep}}. As a by-product of the
previous proof, we can prove the convergence of \eqref{new-rep} by employing
the same method as in \cite[Theorem 4.3]{D6}.

Let $\varphi \in \mathcal{S(}\mathbb{R}^{n})$. By \eqref{supp-cond}, %
\eqref{diff-cond} and \eqref{mom-cond}, and the Taylor expansion of $\varphi 
$ up to order $L$\ with respect to the off-points $x_{Q_{v,m}}$, we obtain
for fixed $v$%
\begin{eqnarray*}
&&\int_{\mathbb{R}^{n}}\sum\limits_{m\in \mathbb{Z}^{n}}\lambda
_{v,m}\varrho _{v,m}(y)\varphi (y)\, dy \\
&=&\int_{\mathbb{R}^{n}}\sum\limits_{m\in \mathbb{Z}^{n}}\lambda
_{v,m}\varrho _{v,m}(y)\Big(\varphi (y)-\sum\limits_{\left\vert \beta
\right\vert \leq L}(y-x_{Q_{v,m}})^{\beta }\frac{D^{\alpha }\varphi
(x_{Q_{v,m}})}{\beta !}\Big)\, dy.
\end{eqnarray*}%
The last factor in the integral can be uniformly estimated from the above by%
\begin{equation*}
c\text{ }2^{-v(L+1)}(1+\left\vert y\right\vert ^{2})^{-M/2}\sup_{x\in 
\mathbb{R}^{n}}(1+\left\vert x\right\vert ^{2})^{M/2}\sum\limits_{\left\vert
\beta \right\vert \leq L+1}\left\vert D^{\alpha }\varphi (x)\right\vert ,
\end{equation*}%
where $M>0$ is at our disposal. Let $0<t<1$ and ${s}(x)=\alpha (x)+\frac{n}{%
p(x)}(t-1)$ be such that $L+1>-s(\cdot )$ $>-\alpha (\cdot )$. Since $%
\varrho _{v,m}$ are $\left[ K,L\right] $-atoms, then for every $S>0$, we
have 
\begin{equation*}
\left\vert \varrho _{v,m}(y)\right\vert \leq c2^{vn/2}\left(
1+2^{v}\left\vert y-x_{Q_{v,m}}\right\vert \right) ^{-S}.
\end{equation*}
Therefore, we find that 
\begin{eqnarray*}
&&\Big|\int_{\mathbb{R}^{n}}\sum\limits_{m\in \mathbb{Z}^{n}}\lambda
_{v,m}\varrho _{v,m}(y)\varphi (y)dy\Big| \\
&\leq &c\text{ }2^{-v(L+1)}\int_{\mathbb{R}^{n}}\sum\limits_{m\in \mathbb{Z}%
^{n}}2^{vn/2}\left\vert \lambda _{v,m}\right\vert \frac{(1+\left\vert
y\right\vert ^{2})^{-M/2}}{\left( 1+2^{v}\left\vert y-x_{Q_{v,m}}\right\vert
\right) ^{S}}\, dy.
\end{eqnarray*}%
Applying Lemma \ref{Conv-est1}, we obtain%
\begin{equation*}
\sum\limits_{m\in \mathbb{Z}^{n}}\left\vert \lambda _{v,m}\right\vert \left(
1+2^{v}\left\vert y-x_{Q_{v,m}}\right\vert \right) ^{-S}\lesssim
\sum\limits_{m\in \mathbb{Z}^{n}}\left\vert \lambda _{v,m}\right\vert \eta
_{v,S}\ast \chi _{v,m}(y).
\end{equation*}%
We split $M$ into $R+S$. Since we have in addition the factor $(1+\left\vert
y\right\vert ^{2})^{-S/2}$, H\"{o}lder's inequality and $(1+\left\vert
y\right\vert ^{2})^{-R/2}\lesssim (1+\left\vert h\right\vert ^{2})^{-R/2}$
give that the term $\left\vert \int_{\mathbb{R}^{n}}\cdot \cdot \cdot
dy\right\vert $ is dominated by%
\begin{eqnarray*}
&&c\ 2^{-v(L+1)}\sum\limits_{h\in \mathbb{Z}^{n}}(1+\left\vert h\right\vert
^{2})^{-R/2}\Big\|\eta _{v,S}\ast \Big(\sum\limits_{m\in \mathbb{Z}%
^{n}}2^{vn/2}\left\vert \lambda _{v,m}\right\vert \chi _{v,m}\Big)\Big\|%
_{p(\cdot )/t} \\
&\lesssim &c\ 2^{-v(L+1+s(x))}\sup_{j\geq 0}\Big\|2^{(s(\cdot
)+n/2)j}\sum\limits_{m\in \mathbb{Z}^{n}}\lambda _{j,m}\chi _{j,m}\Big\|%
_{p(\cdot )/t} \\
&=&c\ 2^{-v(L+1+s(x))}\sup_{j\geq 0}g_{j}
\end{eqnarray*}%
for some $x\in \mathbb{R}^{n}$ and {by taking $R$ large enough}. We see from
H\"{o}lder's inequality that 
\begin{eqnarray*}
g_{j} &=&\frac{1}{\log 2}\int_{2^{-j}}^{2^{1-j}}g_{j}\frac{d\tau }{\tau }%
\leq \frac{1}{\log 2}\int_{2^{-j}}^{2^{1-j}}\big\|\tau ^{-(s(\cdot
)+n/2)j}\sum\limits_{m\in \mathbb{Z}^{n}}\lambda _{j,m}\chi _{j,m}\big\|%
_{p(\cdot )/t}\frac{d\tau }{\tau } \\
&\lesssim &\Big\|\big\|\tau ^{-(s(\cdot )+n/2)j}\sum\limits_{m\in \mathbb{Z}%
^{n}}\lambda _{j,m}\chi _{j,m}\big\|_{p(\cdot )/t}\Big\|_{L^{q(\cdot
)}([2^{-j},2^{1-j}],\frac{d\tau }{\tau })}\big\|1\big\|_{L^{q^{\prime
}(\cdot )}([2^{-j},2^{1-j}],\frac{d\tau }{\tau })}, \\
&\lesssim &\left\Vert \lambda \right\Vert _{\boldsymbol{\tilde{b}}_{p(\cdot
)/t,q(\cdot )}^{s(\cdot )}}.
\end{eqnarray*}%
{The convergence of $\mathrm{\eqref{new-rep}}$ is now clear from the fact
that }$L+1+s(x)>0$ and the embedding 
\begin{equation*}
\boldsymbol{b}_{p(\cdot ),q(\cdot )}^{\alpha (\cdot)} \hookrightarrow 
\boldsymbol{b}_{p(\cdot)/t,q(\cdot )}^{s(\cdot )} \hookrightarrow 
\boldsymbol{\tilde{b}}_{p(\cdot )/t,q(\cdot )}^{s(\cdot )}.
\end{equation*}
The proof is complete.

\begin{ack}
The author would like to thank  Tokio Matsuyama   for the valuable
comments and suggestions.
\end{ack}

Douadi Drihem

M'sila University, Department of Mathematics,

Laboratory of Functional Analysis and Geometry of Spaces,

P.O. Box 166, M'sila 28000, Algeria,

e-mail: \texttt{\ douadidr@yahoo.fr}

\end{document}